\documentclass[10pt,reqno,twoside]{amsart}
\usepackage{amsmath}
\usepackage{amssymb}
\usepackage{epsfig}
\usepackage{color}
\usepackage[initials]{amsrefs}
\usepackage[pdftex,colorlinks,linkcolor=blue,citecolor=blue,urlcolor=blue,
hypertexnames=true,plainpages=false]{hyperref}
    \hypersetup{pdfauthor={Uri Bader, Alex Furman, Alex Gorodnik, Barak Weiss},
                pdftitle ={Rigidity of group actions on homogeneous spaces, III},}

\usepackage[all]{xy}

\numberwithin{equation}{section}

\newcommand{\overto}[1]{{\buildrel{#1}\over\longrightarrow}}

\title[Rigidity of group actions]{Rigidity of group actions on homogeneous spaces, III}

\author{Uri Bader}
\address{Technion, Haifa, Israel 32000}
\email{bader@tx.technion.ac.il}

\author{Alex Furman}
\address{University of Illinois at Chicago, Chicago, IL 60607-7045, USA}
\email{furman@math.uic.edu}

\author{Alex Gorodnik}
\address{University of Bristol \\ Bristol BS8 1TW, UK}
\email{a.gorodnik@bristol.ac.uk}

\author{Barak Weiss}
\address{Ben Gurion University, Be'er Sheva, Israel 84105}
\email{barakw@math.bgu.ac.il}

\newif\ifdraft\drafttrue
\newcommand{\Q}{{\mathbb {Q}}}

\newcommand{\R}{{\mathbb{R}}}

\newcommand{\Z}{{\mathbb{Z}}}

\newcommand{\HH}{{\mathbb{H}}}

\newcommand{\C}{{\mathbb{C}}}

\newcommand{\Ad}{{\operatorname{Ad}}}

\newcommand{\SL}{\operatorname{SL}}

\newcommand{\Lie}{\operatorname{Lie}}

\newcommand{\Map}{{\operatorname{Map}}}
\newcommand{\vol}{{\operatorname{vol}}}

\newcommand{\Aut}{{\rm Aut}}

\newcommand{\Prob}{\operatorname{Prob}}

\newcommand{\sm}{\smallsetminus}

\newcommand\ec{/\!\!/}

\newcommand {\ignore}[1]  {}

\newcommand{\BorelCat}{\mathcal{B}}
\newcommand{\MeasCat}{\mathcal{MCP}}
\newcommand{\Ind}[2]{\operatorname{Ind}_{#1}^{#2}}
\newcommand{\Res}[2]{\operatorname{Res}_{#1}^{#2}}
\newcommand{\acts}{\curvearrowright}
\newcommand{\Rdc}{\operatorname{R}_{\operatorname{dc}}}
\newcommand{\Na}{\operatorname{N}_{\operatorname{a}}}
\newcommand{\Mor}{\operatorname{Mor}}
\newcommand{\Morpmp}{\operatorname{Mor}^{1}}
\newcommand{\MT}[1]{\overline{#1}^{\sc MT}}
\newcommand{\Hilbert}{\mathcal{H}}
\newcommand{\Hypsp}{\mathbf{H}}

\newtheorem{thm}{Theorem}[section]
\newtheorem{lem}[thm]{Lemma}
\newtheorem{prop}[thm]{Proposition}
\newtheorem{cor}[thm]{Corollary}

\theoremstyle{definition}
\newtheorem{Def}[thm]{Definition}
\newtheorem{remark}[thm]{Remark}
\newtheorem{example}[thm]{Example}

\begin{document}

\maketitle

\begin{abstract}
~\\

Consider homogeneous $G/H$ and $G/F$,

for an $S$-algebraic group $G$.

A lattice $\Gamma$ acts on the left

strictly conservatively.

\medskip

The following rigidity results are obtained:

morphisms, factors and joinings defined apriori

 only in the measurable category

 are in fact algebraically constrained.

\medskip

Arguing in an elementary fashion

 we manage to classify

 all the measurable $\Phi$

 commuting with the $\Gamma$-action:

\medskip

  assuming ergodicity, we find

they are algebraically defined.
\end{abstract}

%{\small \tableofcontents }

\section{Introduction}  % (fold)

Flows on homogeneous spaces provide a rich and fruitful source of examples of dynamical systems.
Most of the literature on this subject concerns actions of subgroups $H<G$ on $G/\Gamma$,
where $G$ is a Lie group and $\Gamma<G$ is a lattice.
In this paper we consider a situation where the roles of the subgroups $\Gamma, H<G$ are reversed,
and study actions of a discrete subgroup $\Gamma<G$ on a homogeneous spaces $G/H$.
We shall mostly focus on the situation where $G$ is a Lie group, or a product of algebraic groups over local fields,
$\Gamma<G$ is a lattice, $H<G$ is a closed (algebraic) subgroup, and $G/H$ is equipped with the Haar measure class.
Our aim is to study the classification problem for such objects as:
\begin{enumerate}
	\item measurable $\Gamma$-equivariant maps $G/H\to G/L$,
	\item relatively probability measure-preserving $\Gamma$-quotients of $G/H$,
	\item relatively p.m.p. joinings of $\Gamma$-actions on $G/H$ with $G/L$,
	\item quasi-factors, i.e. $\Gamma$-equivariant maps $G/H\to \Prob(G/H)$,
\end{enumerate}
and seek situations where these \emph{measurable $\Gamma$-objects} (maps, quotients, joinings)
are necessarily \emph{algebraic}, and therefore can be explicitly described.
% We obtain these results in the setting where $G$ is a Lie group or, more generally
% a product of algebraic groups over local fields of zero characteristic, and $H<G$ is a connected
% subgroup and $\Gamma<G$ is a lattice.
These results expand the scope of problems previously studied in this context in \cite{Shalom+Steger} and \cite{Furman}
and generalize most of the results there using new methods.
Before formulating the results, we define the precise framework of this study.
\medskip

\subsection{$\Gamma$-spaces} % (fold)
\label{sub:_gamma_spaces}\hfill{}\\
Let $G$ be a locally compact second countable group.
A \emph{Borel $G$-space} is a standard Borel space $(X,\mathcal{X})$ with a Borel action $G\acts X$,
where the action map $G\times X\to X$ is Borel measurable. Two probability measures on $X$ are called {\em equivalent} if they have the same nullsets.
A \emph{measurable $G$-space} is a Borel $G$-space $X$ equipped with a
probability measure $\mu$, defined on the given Borel $\sigma$-algebra $\mathcal{X}$ on $X$,
which is $G$-quasi-invariant, that is such that $g_*\mu$ is equivalent to $\mu$ for every $g\in G$.
We shall often write $(X,[\mu])$ to emphasize that it is only the measure class $[\mu]$ of $\mu$
that is assumed to be $G$-invariant.
If $(X,[\mu])$ is a measurable $G$-space and $V$ is a Borel $G$-space,
a map $f:X\to V$ is called $G$-\emph{equivariant}, or just a \emph{Borel $G$-map}, if it is Borel
measurable and for every $g\in G$ for $\mu$-a.e. $x\in X$,
\begin{equation}\label{eq:equivariance}
f(g\cdot x)=g\cdot f(x).
\end{equation}
Maps $f,f':X\to V$ that agree $\mu$-a.e. will be identified, and
\[
	\Map_G(X,V)
\]
denotes the set of all equivalence class of $G$-maps. In light of \cite[Appendix B]{Zimmer-book}, any map  in $\Map_G(X,V)$ is equivalent to $f:X \to V$ such that on a conull subset $X_0 \subset X$, (\ref{eq:equivariance}) holds for every $g \in G$ and every $x \in X_0$.
If $G_1$ is a subgroup of $G$ then any $G$-space is automatically a $G_1$-space and $G$-maps are $G_1$-maps, yielding an injective map $\Map_{G}(X,Y) \to \Map_{G_1}(X,Y).$ If this map is surjective we write $\Map_G(X,Y)=\Map_{G_1}(X,Y).$
With the usual abuse of notations $f\in \Map_G(X,V)$ will often mean that $f$ is an actual $G$-map,
rather than an equivalence class of such maps.

A measurable $G$-space $(X,[\mu])$ is \emph{ergodic} if every $G$-invariant measurable subset is $\mu$-null
or $\mu$-co-null.
Let $T$ be a standard Borel space with the trivial $G$-action.
Then $(X,[\mu])$ is ergodic iff every $G$-map $X\to T$ is a constant map, that is $\Map_G(X,T)\cong T$.

% subsection _gamma_spaces (end)

\subsection{Relatively probability measure-preserving factors} % (fold)
\label{sub:relpmp}\hfill{}\\
Let $(X,[\mu])$, $(Y,[\nu])$ be measurable $G$-spaces and $p:X\to Y$ be a $G$-map
so that $[p_*\mu]=[\nu]$.
Such a map will be called $G$-\emph{morphism}, \emph{$G$-quotient}, or
\emph{$G$-factor (map)}. We will be particularly interested in
relatively measure preserving morphisms, which we now define. Recall
that if $X$ is a Borel $G$-space then so is the
%Given such $p:X\to Y$ one may change $\nu$ within its measure class so that $p_*\mu=\nu$. The
space $\Prob(X)$ of Borel probability measures on $X$. Additionally
recall that if $X, Y$ are standard Borel spaces, $\mu$ is a Borel
measure on $X$, $p: X \to Y$ is a Borel
map, and $\nu$ is the pushforward $p_*\mu$, then there is a {\em
  disintegration of $\mu$}, i.e. a
measurable map $Y\to \Prob(X)$, $y\mapsto \mu_y$, such that
\begin{equation}
\label{e:mu-is-average}
	\mu=\int_Y \mu_y\,d\nu(y),\qquad {\rm i.e.}\qquad \mu(E)=\int_Y \mu_y(E)\,d\nu(y)
\end{equation}
for all measurable $E\subset X$, and such that
\begin{equation}
\label{e:mu-y-on-fiber}	
	\mu_y(p^{-1}(\{y\}))=1\qquad{\rm for}\quad\nu{\rm -a.e.\ }y\in Y.
\end{equation}
This map is unique in the sense that any two such maps differ on a
nullset for $\nu$.
We now say that a $G$-morphism $p:(X,[\mu])\to (Y,[\nu])$ is
\emph{relatively probability measure-preserving}
(or relatively p.m.p.) if there is $\mu' \in [\mu]$ such that in the disintegration of $\mu'$ with respect to $p_*\mu'$, the corresponding map $Y\to \Prob(X)$ is  $G$-equivariant, that is,
\begin{equation}
\label{e:equivariant}	
	\mu_{g\cdot y}=g_*\mu_y\qquad(g\in G)
\end{equation}
for $\nu$-a.e. $y\in Y$.
%We note that the property of a $G$-factor map $p:(X,[\mu])\to(Y,[\nu])$ being relatively p.m.p. is independent
%of the choice of $\mu$ within its class.
If the measure class $[\mu]$ contains a $G$-invariant probability measure $\mu_0$, then
$[\nu]$ contains a $G$-invariant probability measure $\nu_0=p_*\mu_0$
and $p$ is relatively p.m.p.
If $G$ is a locally compact second countable group and $H<G$ is a closed subgroup,
then there is a unique invariant measure class on $G/H$ induced by Haar measure on $G$, which we will call the {\em Haar measure class} on $G/H$; the invariant measure class may or may not contain an invariant probability measure. If $L<G$ is another closed subgroup containing $H$, then the $G$-equivariant map $G/H\to G/L$, $gH\mapsto gL$ is relatively p.m.p. iff $L/H$ carries an $L$-invariant probability
measure.

% subsection morphisms_and_relatively_probability_measure_preserving_quotients (end)

\subsection{Relatively p.m.p. joinings and quasi-factors} % (fold)
\label{sub:relatively_p_m_p_joinings_and_quasi_factors}\hfill{}\\
A \emph{relatively p.m.p. joining} of two measurable $\Gamma$-spaces $(X_i,[\mu_i])$, $i=1,2$
is a $G$-quasi-invariant probability measure $\nu$ on $X_1\times X_2$,
so that both projections
\[
	p_i:(X_1\times X_2,\nu)\to (X_i,\mu_i)\qquad(i=1,2)
\]
are relatively p.m.p. maps.
For instance, if
% Any measure class preserving isomorphism $f$ between two measurable $G$-spaces $(X_i,\mu_i)$, % $i=1,2$,
% defines a joining $\nu$ on the graph of $f$:
%\[
	%\nu=\int_{X_1} \delta_x\times\delta_{f(x)}\,d\mu_1(x)
%\]
%where one assumes $\mu_2=f_*\mu_1$.
%If
$p:(X,\mu)\to (Y,\nu)$ is a relatively p.m.p. quotient map, then the pushforward of $\mu$ under the map $x \mapsto (x, p(x))$ has the required properties; this is the
%associated disintegration
%$y\mapsto \mu_y\in \Prob(X)$ defines the measure
%\[
%	\nu=\int_Y \mu_y\times\mu_y\,d\nu(y)
%\]
\emph{relatively independent self-joining} of $(X,\mu)$ associated to $p:X\to Y$.

Given measurable $G$-spaces $(X,\mu)$ and $(Y,\nu)$, we say
that  $(Y,\nu)$ is a \emph{quasi-factor} of $(X,\mu)$ if there is a $G$-map
$\phi\in \Map_G(Y,\Prob(X))$, $y\mapsto \phi_y$, so that
\[
	\mu=\int_Y \phi_y\,d\nu(y).%,\qquad {\rm i.e.}\qquad \mu(E)=\int_Y \phi_y(E)\,d\nu(y)
\]
%for every measurable $E\subset X$.
Thus every relatively p.m.p. quotient map $p:(X,\mu)\to (Y,\nu)$
defines a quasi-factor by disintegration.
%A general quasi-factor is characterized by properties (\ref{e:mu-is-average}) and (\ref{e:equivariant});
%quasi-factors arising from genuine factor maps in addition satisfy (\ref{e:mu-y-on-fiber}).

Furthermore, any relatively p.m.p. joining $\nu$ of $(X_1,\mu_1)$ and $(X_2,\mu_2)$ defines
a pair of quasi-factor maps $X_1\to\Prob(X_2)$, $X_2\to\Prob(X_1)$, via the disintegration
of the projections $p_i:\nu\to\mu_i$.
Hence a classification of all quasi-factors may lead to a classification of relatively p.m.p. joinings;
and relatively p.m.p. factors. For example, if $p:(X,\mu)\to (Y,\nu)$ is a relatively p.m.p.
$G$-factor, and $y\mapsto\mu_y$ is the associated disintegration of $\mu$ with respect to $\nu$,
then
\[
	X\to\Prob(X),\qquad x\mapsto \mu_{p(x)}
\]
is a $G$-quasi-factor.

%If $G'<G$ is a closed subgroup, then all $G$-objects discussed above (Borel $G$-space, measurable $G$-space,
%$G$-maps, relatively p.m.p. factors, joinings, quasi-factors) can be viewed as $G'$-objects.

% subsection relatively_p_m_p_joinings_and_quasi_factors (end)

\subsection{$S$-algebraic groups, subgroups, actions} % (fold)
\label{sub:the_class_of_groups}\hfill{}\\
Let $S$ be a finite set, for each $v\in S$ let $k_v$ be a non-discrete local field of zero characteristic
(i.e., $\R$, $\C$, or a finite extension of $\Q_p$ for some prime $p$),
and $\mathbf{G}_v$ be an algebraic group defined over $k_v$.
Let $G=\prod_{v\in S} \mathbf{G}_v(k_v)$ be the locally compact second countable group formed by
the product of the $k_v$-points $\mathbf{G}_v(k_v)$ of $\mathbf{G}_v$.
We shall use the term \emph{$S$-algebraic group} to describe groups $G$ that arise in this way.
By an \emph{$S$-algebraic subgroup} $H$ of an $S$-algebraic group $G$, we mean any locally compact subgroup
$H=\prod_{v\in S} \mathbf{H}_v(k_v)$ where $\mathbf{H}_v<\mathbf{G}_v$ is a $k_v$-algebraic subgroup
for each $v\in S$.
Similarly, $V$ is an \emph{$S$-variety} if $V=\prod_{v\in S} \mathbf{V}_v(k_v)$ is a product
of $k_v$-points of $k_v$-varieties for $v\in S$.
We shall say that $V$ is an \emph{$S$-algebraic $G$-space} if $G$ is an $S$-algebraic group
and $V$ is an $S$-algebraic variety, equipped with a $G$-action $G\acts V$ associated
to $k_v$-algebraic actions $\mathbf{G}_v\acts \mathbf{V}_v$ for each $v\in S$.
Homogeneous $S$-algebraic spaces are $V=G/H$ where $H<G$ is an $S$-algebraic subgroup.

Any $S$-algebraic space $V$ can be considered as a Borel $G$-space, where the Borel structure
of $V$ is induced by its locally compact topology.
A homogeneous $S$-algebraic $G$-space, $V=G/H$, equipped with the unique $G$-invariant
measure class $[\mu]_{G/H}$ is a measurable $G$-space.
If $M<G$ is a closed (not necessarily $S$-algebraic) subgroup, any $S$-algebraic space
$V$ can be viewed as a Borel $M$-space, and any $S$-algebraic $G$-homogeneous space $G/H$ is a
measurable $M$-space.

Let $G$ be a general locally compact second countable group. A discrete subgroup
$\Gamma<G$ is a \emph{lattice} if $G/\Gamma$ has a $G$-invariant finite measure.
One of the reasons to consider the framework of $S$-algebraic groups,
rather than just algebraic groups over a single field, is that an $S$-algebraic group $G$ may
contain lattices which are not products of lattices in the factors $\mathbf{G}_v(k_v)$.
For example $\Gamma=\SL_n\left(\Z\left[\frac{1}{p}\right]\right)$ is a lattice in the $S$-algebraic
group $G=\SL_n(\R)\times \SL_n(\Q_p)$. A lattice is called {\em irreducible} if it does not have a discrete image in any proper factor of $G$.

Recall that if $H,M$ are closed subgroups in some locally compact group $G$, then $M\acts G/H$
is ergodic iff $H\acts G/M$ is ergodic, where both spaces are considered with the Haar measure classes.
If $\mathbf{G}_v$ are semi-simple $k_v$-groups without compact factors, $\mathbf{H}_v(k_v)$ are non-compact
for every $v\in S$, and $\Gamma<G$ is an irreducible lattice, then $H\acts G/\Gamma$ and $\Gamma\acts G/H$ are ergodic
by Moore's ergodicity theorem.

% subsection the_class_of_groups (end)

\subsection{Statements of the main results} % (fold)
\label{sub:equivariant_maps}\hfill{}\\
We shall now present the main results for actions of a lattice $\Gamma$ in an $S$-algebraic group $G$
on $G/H$ where $H<G$ is an $S$-algebraic subgroup.
%We consider $G/H$ equipped with the $G$-invariant measure class (Haar measure class)
%as a measurable $\Gamma$-space.
We shall assume ergodicity of $\Gamma\acts G/H$ (equivalently $H\acts G/\Gamma$).
Let us first state the results under the following simplifying assumption
\begin{itemize}
	\item[(*)] $H$ has no non-trivial $S$-algebraic compact factor group.
\end{itemize}
Hereafter for any group action $G\acts S$ we denote by $S^G$ the set of $G$-fixed points in $S$.
\begin{thm}[Equivariant Borel maps]
	\label{T:equiv-Borel-maps}\hfill{}\\
	Let $G$ be an $S$-algebraic group, $H<G$ an $S$-algebraic subgroup, $\Gamma< G$ a lattice,
	and $G/H$ an ergodic $\Gamma$-space.
	Assume (*).
	Let $V$ be an $S$-algebraic $G$-space, viewed as a Borel $\Gamma$-space. Then
	\[
		\Map_\Gamma(G/H,V)=\Map_G(G/H,V)\cong V^H
	\]
	where $v\in V^H$ corresponds to the $G$-map $f_v(gH)=g\cdot v$.
\end{thm}
As an immediate corollary we obtain measure-theoretic rigidity results.
Given a measurable $\Gamma$-space $(X,[\mu])$ denote by $\Aut(X,[\mu])$ the group
of measure-class automorphisms of $X$, up to null sets,
and let $\Aut_\Gamma(X,[\mu])$ denote the subgroup of $\Gamma$-equivariant ones.

We denote conjugation by $h^g=g^{-1}hg$, $H^g=g^{-1}Hg$, and normalizers by
\[
	\mathcal{N}_G(H)=\{ g\in G \mid H^g=H\}.
\]
Recall that $\mathcal{N}_G(H)/H$ is $\Aut_G(G/H)$ -- the group of $G$-equivariant bijections of
$G/H$ as a set. Here $nH\in \mathcal{N}_G(H)/H$ acts on $G/H$ by $gH\mapsto gnH$.
\begin{cor}\label{C:isom-centralizer}\hfill{}\\
	Let $G$ be an $S$-algebraic group, $H,H_1,H_2<G$ be $S$-algebraic subgroups satisfying (*), $\Gamma<G$ a lattice whose action on $G/H_i$ is ergodic. Then
	\begin{enumerate}
		\item
		$\Map_\Gamma(G/H_1,G/H_2)=\Map_G(G/H_1,G/H_2)$ and all such maps are given by $\ gH_1\mapsto gg_0H_2$,
		where $H_1^{g_0}<H_2$.
		\item
		The spaces $G/H_1$ and $G/H_2$ are isomorphic as ergodic $\Gamma$-spaces iff
		they are algebraically isomorphic, i.e., $H_1^{g_0}=H_2$ for some $g_0\in G$.
		\item
		$\Aut_\Gamma(G/H)=\Aut_G(G/H)\cong \mathcal{N}_G(H)/H$.
	\end{enumerate}
\end{cor}

\medskip
Given an ergodic $\Gamma$-space $G/H$ as above, it is possible to
classify all of its relatively p.m.p.
$\Gamma$-factors.
\begin{thm}[Relatively p.m.p. factors]
	\label{T:factors}\hfill{}\\
	Let $\Gamma, H< G$ be as in Theorem~\ref{T:equiv-Borel-maps},
        and let $p:G/H\to Y$ be a relatively p.m.p.
	$\Gamma$-factor.
	
	Then $(Y,\nu)$ is a relatively p.m.p. $G$-factor of $G/H$.
	More precisely, there is a closed subgroup $H\triangleleft L<G$ with
	$K=L/H$ compact, so that $Y\cong G/L$ and
	\[
		p:G/H\to Y\cong G/L\qquad\text{is\ given\ by}\qquad gH\mapsto gL.
	\]
\end{thm}

\medskip

The following result shows that under assumption (*), ergodic
measurable $\Gamma$-spaces $G/H_1$ and $G/H_2$
have no non-trivial relatively p.m.p. joinings, unless $G/H_1\cong G/H_2$.

\begin{thm}[Relatively p.m.p. joinings]
	\label{T:joinings}\hfill{}\\
	Let $\Gamma, H_1,H_2< G$ be as in Corollary~\ref{C:isom-centralizer}.
	Then the ergodic $\Gamma$-spaces $G/H_1$ and $G/H_2$ admit a relatively p.m.p. joining
	iff $H_1$ and $H_2$ are conjugate in $G$.
\end{thm}

\medskip

%In light of the discussion in \S\ref{sub:relatively_p_m_p_joinings_and_quasi_factors}, Theorems %\ref{T:factors} and \ref{T:joinings} follows immediately from the following classification of quasi-factors.

\begin{thm}[Quasi-factors]
	\label{T:quasi-factors}\hfill{}\\
	Let $\Gamma, H< G$ and $V$ be as in Theorem~\ref{T:equiv-Borel-maps}, and let $\phi:G/H\to \Prob(V)$
	be a $\Gamma$-quasi-factor, where $V$ is equipped with the $\Gamma$-quasi-invariant
	measure
	\[
	\nu=\int_{G/H} \phi_{gH}\,d\mu_{G/H}.
	\]
	Then $V^H\ne\emptyset$ and $\phi_{gH}=g\cdot \nu_0$
	for some fixed $\nu_0\in \Prob(V^H)$.
\end{thm}

\medskip

In short, the results above show that certain classes of measurable $\Gamma$-maps
on a $G$-homogeneous space $G/H$ are necessarily $G$-maps, and
therefore can be explicitly described using transitivity
of the $G$-action.
These results depend on the assumption that the $S$-algebraic subgroup
$H<G$ has no $S$-algebraic compact factors.
In the case of a general $S$-algebraic subgroup $H<G$, measurable
$\Gamma$-maps as above need not be $G$-maps.
We shall show, however, that such maps are $M$-maps, where $M$ is some
closed cocompact subgroup of $G$ containing $\Gamma$
and acting transitively on $G/H$. Moreover, $M$ is a \emph{fat
  complement} of $H$ in $G$ as described in
Definition~\ref{d:fat}.

\begin{thm}[Quasi-factors, general case]
	\label{T:general-quasi-factors}\hfill{}\\
	Let $G$ be an $S$-algebraic group, $H<G$ an $S$-algebraic
        subgroup, $\Gamma<G$ a lattice and assume that
	$G/H$ is an ergodic $\Gamma$-space.
	Let $V$ be an $S$-algebraic $G$-space with a
        $\Gamma$-quasi-factor $\phi:G/H\to \Prob(V)$,
	where $V$ is equipped with the $\Gamma$-quasi-invariant measure
	\[
		\nu=\int_{G/H} \phi_{gH}\,d\mu_{G/H}.
	\]
	Then there exists a closed cocompact subgroup $M<G$,
        containing $\Gamma$, acting transitively on $G/H$
	(and being a fat complement of $H$ as in Definition~\ref{d:fat}),
	so that $V^{M\cap H}\ne\emptyset$, and a probability measure $\nu_0$ supported on $V^{M\cap H}$, such that for a.e. $m \in M$,
	\[
	\phi_{mH}=m\cdot\nu_0. %\qquad (m\in M)
	\]
Moreover, if $G$ satisfies (*) then $M$ has finite index in $G$.
%	where $\nu_0$ is a probability measure supported on $V^{M\cap H}$.
\end{thm}

\medskip

When $M$ acts transitively on $G/H$, for any $g \in G$ there is $m \in M$ such that $gH=mH$, and any two such elements of $M$ differ by right-multiplication by an element of $M \cap H$. Therefore for $v \in V^{M\cap H}$, the map $f_v(gH)=m\cdot v$ is well-defined and is an $M$-map; and clearly all $M$-maps $G/H \to V$ arise in this way.

\begin{cor}[Equivariant maps, general case]
	\label{C:general-maps}\hfill{}\\
	Let $\Gamma, H<G$ and $V$ be as above. Then there exists $M<G$ as in Theorem~\ref{T:general-quasi-factors}, so that
	\[
		\Map_\Gamma(G/H,V)=\Map_M(G/H,V)\cong V^{M\cap H}.
	\]	
	%where $v\in V^{M\cap H}$ corresponds to the $M$-map $f_v(mH)=m\cdot v$ ($m\in M$).
\end{cor}

\begin{cor}[Relatively p.m.p. factors, general case]
	\label{C:general-factors}\hfill{}\\
	Let $\Gamma, H<G$ be as above, and $p:G/H\to Y$ be a relatively p.m.p. $\Gamma$-factor.
		Then there exists $M<G$ as in Theorem~\ref{T:general-quasi-factors}, a closed subgroup $M\cap H\triangleleft L<M$
	so that $L/(M\cap H)$ has a finite $L$-invariant measure, $Y\cong M/L$, and
	\[
		p:G/H\to Y\cong M/L\qquad \text{is\ given\ by}\qquad mH\mapsto mL\qquad (m\in M).
	\]
%SOMETHING IS FISHY HERE. IS Y EQUAL TO L/M CAP H OR M/L?
 \end{cor}

% subsection statements_of_the_main_results (end)

\subsection{Previous results} % (fold)
\label{sub:previous}
The general inspiration for the questions considered here is the pioneering work of Marina Ratner
\cite{r1, r2, r3} where questions of measurable isomorphism,
classification of factors and joinings where
studied for actions of unipotent subgroup $H<G$ on $G/\Gamma$, where
$\Gamma$ is a lattice in a Lie group $G$
(these and more general results can also be deduced from the general
Ratner's theorem \cite{r4}, see also \cite{Wi}).
The assumption that $H<G$ is unipotent is very important for these results, as they fail for
the diagonal subgroup $A<\SL_2(\R)$.
It should also be emphasized that these are probability measure-preserving actions.

Questions of measure-theoretic rigidity for the $\Gamma$-action on infinite measure homogeneous spaces $G/H$
were addressed by Shalom and Steger in \cite{Shalom+Steger}.
In this (unpublished) work the authors obtained a number of rigidity results using representation theoretic techniques;
including rigidity of lattices in $\SL_2(\R)$ acting on $\R^2$, classification of relatively p.m.p.
$\Gamma$-factors of $\R^n$ where $\Gamma<\SL_n(\R)$ is a lattice, etc.

In \cite{Furman} further results were obtained for $\Gamma$-actions on $G/H$ using purely dynamical methods
(\emph{alignment} property). These results included classification of centralizers, relatively p.m.p. quotients,
and joinings for a particular type of homogeneous space $G/H$. In particular, in \cite{Furman}
the group $G$ was assumed to be semi-simple, $H<G$ to be ``super-spherical",
and $G/H$ to carry an infinite $G$-invariant measure.

The present paper provides a more systematic study of $\Gamma$-actions on $G/H$, reproving most of the
results of \cite{Furman} and \cite{Shalom+Steger} and significantly expanding the scope of the spaces $G/H$ and of the questions.
In particular, we impose no restrictions on $G$ and $H$ except for being $S$-algebraic. For instance, even the very special case $G=\SL_2(\R), \, H_1=H_2=A$ (the group of diagonal matrices) of Corollary \ref{C:isom-centralizer} is new.
The new questions include equivariant maps to $S$-algebraic $G$-spaces (Theorems~\ref{T:equiv-Borel-maps}
and Corollary \ref{C:general-maps})
and classification of quasi-factors (Theorems~\ref{T:quasi-factors} and \ref{T:general-quasi-factors}).
%We learned about the concept of quasi-factors from Glasner's \cite{G}.

Our results bear some similarity to Margulis's factor theorem \cite{Margulis-factor} that asserts
that all measurable $\Gamma$-equivariant
quotients of $G/P$ are $G$-equivariant, and therefore have the form $G/P\to G/Q$, $gP\mapsto gQ$, where $P<Q<G$
are parabolic subgroups.
However, the similarity is only superficial, as the context, phenomenology, and the idea of the proof are completely
different.
Margulis's factor theorem concerns \emph{higher rank} semi-simple group $G$ and a \emph{parabolic} subgroup $P<G$,
but the quotient maps are not assumed to be (and never are) relatively p.m.p.
Similarly for related work on factor theorems by Zimmer \cite{Zimmer-factor}, Nevo-Zimmer \cite{Nevo+Zimmer},
and Bader-Shalom \cite{Bader+Shalom}.

% subsection background (end)

\subsection{Some ideas in the proofs} % (fold)
\label{sub:some_ideas_in_the_proofs}\hfill{}\\
If $H$ is a subgroup of $G$, one may {\em restrict} any $G$-space to obtain an $H$-space, and there is a complementary operation of {\em inducing} an $H$-space to obtain a $G$-space. These operations have been extensively studied in other contexts in representation theory and ergodic theory. In \S \ref{s:frob} we adapt them to our framework and establish an analogue of Frobenius reciprocity, which allows us to convert questions about $\Gamma$-maps, where $\Gamma$ is a lattice, to questions about $H$-maps, where $H$ is algebraic. Then, to study algebraic actions, we apply generalizations of the classical Borel density theorem \cite{Borel-density}, which says that if $G$ is a semi-simple $S$-algebraic
group with no compact factors and $\Gamma<G$ is a lattice, then $\Gamma$ is Zariski dense in $G$.
This theorem was generalized by a number of authors
(see \cite{Dani,Furst-density,Shalom-density, Wa1}), and it was realized that this phenomenon is related to
the classification of $G$-invariant probability measures on $V$, where $V$ is an $S$-algebraic $G$-space,
leading to a more abstract version which says that every such measure must be supported
on the set of fixed points, i.e.
\[
	\Prob(V)^{G}=\Prob(V^{G}).
\]
We obtain one such generalization (Theorem \ref{th:BDT}) as a straightforward corollary of previous work.
Frobenius reciprocity and Borel density are already enough to imply our results in case $H$ satisfies (*).
To obtain our results in the general case, we need a more general version of Borel density (Theorem \ref{th:admiss}), which we call {\em relative Borel density.}
Theorem \ref{th:admiss} is the main technical innovation in our paper. Its proof relies on classical arguments, along with a detailed study of the Mautner envelope and a general version of ergodic decomposition.

% subsection some_ideas_in_the_proofs (end)

\subsection{Some additional results} % (fold)
\label{sub:some_additional_results}\hfill{}\\
Our discussion so far has focused on actions of \emph{lattices} $\Gamma<G$ on homogeneous
spaces $G/H$ where $H$ is an $S$-algebraic subgroup of an $S$-algebraic group $G$.
In fact, we have only used the fact that $G/\Gamma$ carries a finite $G$-invariant
measure and the results apply verbatim to actions of not necessarily discrete closed subgroups $L$ of $G$
provided $G/L$ carries a \emph{finite} $G$-invariant measure. Such subgroups $L<G$ are said to be of
\emph{finite covolume} in $G$.

However, the finite covolume assumption is not necessary for the results as above to hold;
for it is primarily used in arguments involving Borel's density theorem,
where existence of finite invariant measure provides recurrence via Poincar\'e recurrence theorem.
In some situations one can exploit recurrence phenomena for actions of subgroups $\Gamma<G$
of \emph{infinite covolume} on some $G/H$ to prove results analogous to
\cite[Theorem B]{Furman}.
For an example of such a situation,
let $G$ be
%a rank one Lie group  ${\rm SO}(n,1)$, ${\rm SU}(n,1)$, ${\rm Sp}(n,1)$, or ${\rm F}_{4\,(-20)}$,
%i.e.
the group of orientation preserving isometries of a symmetric space of rank one
$\Hypsp^n_K$, with $K=\R$, or $K=\mathbb{C}$. %, $\mathbb{H}$, or $\mathbb{O}$ and $n=2$.
Such $G$ acts transitively on the space of pairs of distinct points at
the boundary $\partial\Hypsp_K^n$
\[
	\partial^{2}\Hypsp^n_K=\partial \Hypsp^n_K\times\partial\Hypsp^n_K\smallsetminus\Delta.
\]
Hence the latter space can be identified with $G/A$ where $A$ is the Cartan subgroup of $G$.
We shall consider $\partial^{2}\Hypsp^n_K\cong G/A$ equipped with the $G$-invariant measure class.
\begin{thm}\label{cor: non-lattices1}
Let $n\ge 2$ and $G={\rm Isom}_+(\Hypsp^n_K)$, and $\Gamma<G$ be a discrete subgroup
with an ergodic action on $G/A\cong \partial^{2}\Hypsp^n_K$.
Then
\[
	\Map_\Gamma(G/A, G/A)=\Map_G(G/A,G/A)=\Aut_G(G/A)\cong\mathcal{N}_G(A)/A=\{1,F\}
\]
where $F$ denotes the flip $F:(x,y)\mapsto (y,x)$ on $\partial^{2}\Hypsp^n_K$.
\end{thm}
In the statement above the source $G/A$ is considered as a measurable
$G$-space (resp. $\Gamma$-space),
and the target $G/A$ can be viewed either as a measurable or merely as
a Borel $G$-space (resp. $\Gamma$-space).
The ergodicity assumption in the theorem is equivalent to the ergodicity of $A\acts G/\Gamma$.
If $\Gamma<G$ is a lattice, ergodicity is guaranteed by Moore's theorem. Thus this
case is covered by Theorem~\ref{T:equiv-Borel-maps}.
On the other hand,
%However, if $K=\R$ or $\C$ (and only in these cases, see Remark~\ref{R:T-infcovol} below)
there exist discrete subgroups $\Gamma<G$ with infinite covolume with $\Gamma\acts G/A$ ergodic.
For instance,  if $\Gamma\triangleleft\Lambda$ where
$\Lambda<G$ is a cocompact lattice and
$\Lambda/\Gamma\simeq \mathbb{Z}$ or $\mathbb{Z}^2$ then $A\acts
G/\Gamma$ is ergodic (\cite{Rees}, \cite{Guiv}).
Examples of such infinite covolume $\Gamma<G$ can be constructed in
all groups $G={\rm SO}(n,1)$ (that is, $K=\R$) and ${\rm SU}(n,1)$
($K=\C$).
\begin{remark}\label{R:T-infcovol}
Theorem \ref{cor: non-lattices1} is also valid in case $G={\rm Sp}(n,1)$ (that is,
$K=\mathbb{H}$) or $G={\rm F}_{4\,(-20)}$ ($K$ the octonion numbers),
but gives no examples which are not covered by Theorem \ref{T:equiv-Borel-maps}.
This is due to the fact that in these cases, $G$ has property (T) and
ergodicity of $A\acts G/\Gamma$ implies $\vol(G/\Gamma)<+\infty$.
%% BW: need a reference for this.
%See
%\combarak{give precise reference -- Yehuda TATA?}
%
%for $G={\rm Sp}(n,1)$ and $G={\rm F}_{4\,(-20)}$, can be deduced from property (T) satisfied
%by these groups.
%More precisely, one can show that if $\Gamma$ has infinite covolume,
%for a.e. $x\in G/\Gamma$ the $A$-orbit escapes to infinity, in fact, escapes with a certain
%speed (cf. \cite{}), ruling out ergodicity.
\end{remark}

\bigskip

Another possible modification of our basic setup concerns the assumptions on $H<G$,
namely instead of assuming $H$ to be an $S$-algebraic subgroup, one might consider
a general closed subgroup, or a discrete subgroup as an extreme case.
More specifically, consider the space of $\Gamma$-equivariant measurable maps
\[
	\Map_\Gamma(G/\Lambda,G/\Delta)
\]
where $\Lambda,\Delta<G$ are discrete subgroups and $\Gamma$ is a lattice in $G$.
As before, the source space $G/\Lambda$ is equipped with the Haar measure class
and is viewed as an ergodic $\Gamma$-space, while the target space $G/\Delta$ may be
viewed either as a measurable $\Gamma$-space, or as a Borel $\Gamma$-space.
%Observe that in this setup there are two natural cases
% \begin{enumerate}
% 	\item
% 	$\Gamma$-maps are actually $G$-maps.
% 	Note that $G$-maps $G/\Lambda\to G/\Delta$ exist if and only if $\Lambda$ is conjugate into $\Delta$;
% 	in fact, $G$-maps $f:G/\Lambda\to G/\Delta$ have the form
% 	\[
% 		f_a(g\Lambda)=ga\Delta,\qquad\text{where}\qquad a^{-1}\Lambda a < \Delta.
% 	\]
% 	\item
% 	constant maps $c_a(g\Lambda)=a\Delta$.
% 	Note that such maps are $\Gamma$-equivariant iff $\Gamma$ is conjugate into $\Delta$;
% 	in fact, constant $\Gamma$-maps  $c:G/\Lambda\to G/\Delta$ have the form
% 	\[
% 		c_a(g\Lambda)=a\Delta,\qquad\text{where}\qquad a^{-1}\Gamma a < \Delta.
% 	\]
% \end{enumerate}
% Here are two situations where

\begin{thm}\label{T:bfgw2}
Let $G$ be a connected non-compact simple real Lie group, $\Gamma<G$ a lattice, $\Lambda, \Delta<G$ discrete subgroups,
with $\Lambda$ Zariski dense in $G$.
%Assume that $\Gamma\acts G/\Lambda$ is ergodic with respect to the Haar measure class %% BW %ergodicity is immediate from Howe Moore, and
In addition assume one of the following, either:
\begin{itemize}
	\item[(RS)] %BW: made a generalization here as Shah's results are more general.
	$\Lambda$ is a lattice in a subgroup of $G$ generated by unipotent elements, or
	\item[(BQ)]
	$\Delta$ is a contained in a lattice $\Delta_0<G$.
\end{itemize}
Then any element of $\Map_\Gamma(G/\Lambda, G/\Delta)$ is either:
\begin{enumerate}
	\item a $G$-map $g\Lambda\mapsto ga\Delta$, where $a\in G$ is such that $a^{-1}\Lambda a<\Delta$, or
	\item a constant map $g\Lambda\mapsto a\Delta$, where $a\in G$ is such that $a^{-1}\Gamma a <\Delta$.
\end{enumerate} 	
\end{thm}

\ignore{
Here is another application in which the existence of a $\Gamma$-map
restricts the ambient group containing $\Gamma$. A subgroup $H$ of a
simple rank-one Lie group $G$ is called {\em horospherical} if $G/H$
for some
$\Ad$-split $x \in G$ we have
$$H =
\{g \in G: x^n g x^{-n} \to \mathrm{id} \};$$
in this case $G/H$ is identified with the space of horospheres in the
symmetric space corresponding to $G$.
\begin{thm}\label{thm: non-lattices2}
For $i=1,2$ let $G_i$ be a non-compact rank one simple Lie group with trivial
center, let $H_i$ be horospherical, and let $\Gamma_i$ be
discrete subgroups acting ergodically on $G_i/H_i$ which are
isomorphic as abstract groups via an isomorphism $\tau$. Suppose $c:
G_1/H_1 \to G_2/H_2$ is a measurable map satisfying $c(\gamma x) =
\tau(\gamma) c(x)$ for all $\gamma \in \Gamma_1$ and a.e. $x \in
G_1/H_1.$ Then
$\tau$ extends to a an isomorphism $\bar{\tau}: G_1 \to G_2$ such that
$\bar{\tau}(H_1) \subset H_2$, and (after changing $c$ on a conull subset)
for some $g_2 \in N_{G_2}(H_2)$ we have
$c(gH_1) =  \bar{\tau}(g)g_2H_2$ for all $g \in G_1$.
\end{thm}

This strengthens results of \cite{Furman}; note for instance that if
$\Lambda$ is a
cocompact lattice in $G_1$ and $\Gamma$ is a normal subgroup
of $\Lambda$ with $\Lambda/ \Gamma \cong \Z^d$, then the
hypotheses of Theorem \ref{thm: non-lattices2} are satisfied for any
$d$, but those of \cite[Thm. B]{Furman} are only satisfied if $d \leq
2.$
}
\medskip

% subsection some_additional_results (end)

\subsection{Organization of the paper}
%Adhering to a well-established
%tradition, we begin with a title, abstract and introduction, and end
%with a list of
%references and mailing addresses (both regular and electronic).

In Section \ref{s:frob} we introduce the category of measurable $G$-spaces
and establish a version of Frobenius reciprocity that will play a crucial role
in the proof. Then in Section \ref{s:borel_density} we discuss the Borel density
theorem and in Section \ref{sec:proofs_of_the_basic_results}, give a self-contained proof of our main results under the assumption
that $H$ has no compact algebraic quotients. Sections \ref{s:frob}-\ref{sec:proofs_of_the_basic_results} exhibit our main ideas, and quickly prove a substantial part of our results, avoiding technicalities.

The proof in full generality
requires further preparation.
In Section \ref{s:mautner} we introduce
the notions of the Mautner envelope and the f.i.-algebraic kernel, and
establish a general version of the Mautner property (Theorem \ref{th:mau}),
which is of independent interest. In Section \ref{s:rel} we discuss fat complements
%the space of ergodic components,
%%BW: do not understand what theorem is being called the culmination.
%prove a general rigidity result
%(Theorem \ref{t:hw}), which can be considered the culmination of the paper,
and deduce the relative Borel density theorem.
%This section also contains discussion of fat complements.
The proofs of the remaining results are given in Section \ref{s:proof}.
As we show they follow easily from the relative Borel density theorem.
In Section \ref{s:ex}, we give several examples to demonstrate that
the assumptions in the main theorems are essential. Finally, in an appendix to the paper, we summarize well-known results on the space of ergodic components and deduce some corollaries.

%Along
%the way we acknowledge support of various funding agencies,
%describe the organization of the paper, and prove our
%results. Departing from tradition, we do not specify which of
%the sections may be omitted on a first reading.
%The appendix contains
%results which are presumably well-known
%to the experts who know them, but for which a suitable reference could
%not be found.

\medskip

\subsection{Acknowledgments}
We would like to thank the organizers of the meeting {\em Rigidity, dynamics
and groups actions} in July 2005 at Banff International
Research Station, during which this project was conceived. Most of
the research described in this paper was conducted at Banff in June
2006 under the `research in teams'
program. We would like to thank BIRS for a great
time and superb
working conditions.
%We also thank Caltech
%for Advanced Studies in Mathematics at Ben Gurion
%University, \ldots for funding mutual visits.

This research was supported by ISF grants 584/04 and 704/08,
BSF grant 2008267, NSF grant DMS 0905977, EPSRC grant EP/H000091/1, and ERC grant 239606.
% section Introduction (end)

\section{Frobenius reciprocity}\label{s:frob} % (fold)
Let $G$ be a locally compact second countable group. Denote by
$\BorelCat_G$ the category of all Borel $G$-spaces. The objects of this category
are Borel actions on standard Borel spaces $G\acts X$, and morphisms
are Borel $G$-maps.

We shall denote by $\MeasCat_G$ the category of measurable $G$-spaces.
The objects of this category are measure class preserving actions of $G$ on standard
probability spaces $G\acts (X,[\mu])$, and the morphisms are $G$-morphisms $p:X
\to Y$ for which $[\nu]=[p_*\mu]$.
%For morphisms one may consider all measure-class preserving $G$-maps
%$p:(X,[\mu])\to (Y,[\nu])$, that is measurable maps $p:X\to Y$
%preserving the class of $\mu$, so that
%\[
%%	p_*\mu\sim \nu,\qquad \text{and}\qquad
%	\forall g\in G\quad p(g\cdot x)=g\cdot p(x)\quad \mu\text{-a.e.}
%\]
We identify any two $G$-morphisms %such maps $p,p':X\to Y$
that agree $\mu$-a.e. and write $\Mor_G(X,Y)$
for the set of equivalence classes.
We shall distinguish a subset of morphisms:
\[
	\Morpmp_G(X,Y)=\{ p\in \Mor_G(X,Y) : p\ \text{relatively\ p.m.p.}\}.
\]
Recall that given $X\in \MeasCat_G$ and $Z\in\BorelCat_G$ we denote by
\[
	\Map_G(X,Z)
\]
the set of equivalence classes of Borel $G$-equivariant maps
$\Phi:X\to Z$.
\medskip

Let $H$ be a closed subgroup of %a locally compact second countable group
$G$.
Then every Borel $G$-space is a Borel $H$-space, and every
Borel $G$-map is a Borel $H$-map. Similarly for the measurable category.
Hence, we have the {\it restriction functor}
\[
  \Res{H}{G}:\BorelCat_G\to \BorelCat_H,\qquad\Res{H}{G}:\MeasCat_G\to \MeasCat_H.
\]
There is also a natural construction of the {\it induction functor} (cf. \cite[\S4.2]{Zimmer-book})
\[
	\Ind{H}{G}:\BorelCat_H\to \BorelCat_G,\qquad\Ind{H}{G}:\MeasCat_H\to \MeasCat_G
\]
defined as follows. Given a Borel $H$-space $X$, consider the factor map $\pi:G\times X\to (G\times X)/\!\!\sim$
where $\sim$ is the equivalence relation $(g,x)\sim (gh,h^{-1}x)$ for $g\in G$, $h\in H$, and $x\in X$.
As a measure space, $(G \times X)/\!\!\sim$ is isomorphic to $G/H \times X$ so is a standard Borel space. The $G$-action on $(G\times X)/\!\!\sim$
is given by $g[g_1,x]=[gg_1,x]$ for $g,g_1\in G$ and $x\in X$.
This defines the induced $G$-space $\Ind{H}{G}(X)$.
Given an $H$-map $p:X\to Y$ between Borel $H$-spaces, define the induced map
\[
	\Ind{H}{G}(p):\Ind{H}{G}(X)\to \Ind{H}{G}(Y)
\]
by $[g,x]\mapsto [g,p(x)]$ for $g\in G$ and $x\in X$.
In the measurable category one applies the same constructions to the underlying
Borel spaces. As for the measures, let $\lambda$ be a probability measure on $G$
equivalent to Haar measure on $G$.
We equip $(G\times X)/\!\!\sim$ with the probability measure obtained by pushing forward $\lambda \times \mu$ under the quotient map $G \times X \to (G \times X)/\!\!\sim$.
The $G$-action preserves the induced measure class on $\Ind{H}{G}(X)$. We leave the proof of the following straightforward statement to the reader.
\begin{lem}
	Let $X,Y \in \MeasCat_H$. Then
	\begin{enumerate}
		\item If $p\in \Mor_H(X,Y)$ then $\Ind{H}{G}(p)\in \Mor_G(\Ind{H}{G}(X),\Ind{H}{G}(Y))$,
		\item If $p\in \Morpmp_H(X,Y)$ then $\Ind{H}{G}(p)\in \Morpmp_G(\Ind{H}{G}(X),\Ind{H}{G}(Y))$.
	\end{enumerate}
\end{lem}
\qed

Next we show that the functors $\Res{H}{G}$ and $\Ind{H}{G}$ are formally adjoint.
\begin{prop}[Frobenius reciprocity] \label{p:frob}\hfill{}\\
Let $X\in  \MeasCat_H$, $Y\in  \MeasCat_G$ and $Z\in\BorelCat_G$.
Then there are natural bijections
\begin{eqnarray*}
	\Mor_H(X,\Res{H}{G}(Y))&\simeq& \Mor_G(\Ind{H}{G}(X),Y),\\
	\Morpmp_H(X,\Res{H}{G}(Y))&\simeq& \Morpmp_G(\Ind{H}{G}(X),Y),\\
	\Map_H(X,\Res{H}{G}(Z))&\simeq& \Map_G(\Ind{H}{G}(X),Z),
\end{eqnarray*}
in all three cases given by
\begin{equation}\label{eq:phi}
	\Phi \longmapsto ([g,x]\mapsto g\Phi(x)).
\end{equation}
\end{prop}
\begin{proof}
Let $\Phi: X\to Y$ be an $H$-map, where $Y$ is a $G$-space. It follows
from \cite[B.5]{Zimmer-book} that after modifying
$\Phi$ on a null set, there exists an $H$-invariant Borel subset $X_0\subset X$
of full measure such that $\Phi|_{X_0}$ is $H$-equivariant.
This shows that formula (\ref{eq:phi}) gives a well-defined map
from $\Mor_H(X,\Res{H}{G}(Y))$ to $\Mor_G(\Ind{H}{G}(X),Y)$.

Now we construct its inverse. Every element in $\Mor_G(\Ind{H}{G}(X),Y)$
can be lifted to a map $\Psi: G\times X \to Y$ such that
\begin{equation}\label{eq:equiv}
	\Psi(gg',x)=g\Psi(g',x)\quad\text{and}\quad \Psi(g'h^{-1},h\cdot x)=\Psi(g',x)
\end{equation}
for every $g\in G$, $h\in H$, and almost every $(g',x)\in G\times X$.
Moreover, this correspondence respects the measure-class-preserving property.
Applying \cite[B.5]{Zimmer-book} to the action of $G\times H$ on $G\times X$ given by
\[
	(g,h).(g',x)=(gg'h^{-1}, h\cdot x),
\]
we conclude that after modifying $\Psi$ on a set of measure zero, we may assume
that there exists an $H$-invariant Borel subset $X_0$ of $X$ with full measure
such that (\ref{eq:equiv}) holds for all $(g',x)\in G\times X_0$.
Then  $\Phi(x)=\Psi(e,x)$, $x\in X$, is an $H$-map and it is clear that it defines the inverse
to the map (\ref{eq:phi}).
\end{proof}

The induced space $\Ind{H}{G}(X)$ is easy to describe for $G$-spaces.

\begin{prop}\label{p:g_space}
For every $X\in \MeasCat_G$,
\[
	\Ind{H}{G}(\Res{H}{G}(X))\simeq G/H\times X
\]
where the latter space is equipped with the product measure class and the  diagonal action.
%Moreover, this isomorphism is measure-class-preserving.
\end{prop}

\begin{proof}
The map $G\times X\to G\times X: (g,x)\mapsto (g,g\cdot x)$ induces an isomorphism of the
spaces in question.
\end{proof}

Now we establish several corollaries that give correspondences among sets of
maps in different categories.

\begin{cor}\label{p:main}
Let $X,Y\in\MeasCat_G$, $Z\in\BorelCat_G$, and $H<G$ a closed subgroup.
Then under the bijections given by Propositions \ref{p:frob} and \ref{p:g_space},
\begin{align}\label{eq:bij}
	\Mor_H(X,Y) &\simeq \Mor_G(G/H\times X,Y),\nonumber\\
	\Morpmp_H(X,Y) &\simeq \Morpmp_G(G/H\times X,Y),\\
	\Map_H(X,Z) &\simeq \Map_G(G/H\times X,Z).\nonumber
\end{align}
%Moreover, the subset
%of $M$-maps correspond to the subset of maps $G/\Gamma\times X\to Y$
%that factor through the natural factor map $G/\Gamma\times X\to G/M\times X$.
\end{cor}

\begin{proof}
%The bijection (\ref{eq:bij}) is given by
%$$
%\Phi\;\; \longleftrightarrow \;\; \Psi(z,x)=z\Phi(z^{-1}x),\,\, z\in G/\Gamma,\, x\in X.
%$$
The claim follows from Propositions \ref{p:frob} and \ref{p:g_space}.
\end{proof}

\medskip

\begin{cor} \label{p:duality}
Let $H,L<G$ be closed subgroups, $Y\in\MeasCat_G$, $Z\in\BorelCat_G$.
Then there are natural bijections
\begin{align}\label{eq:bij}
	\Mor_L(G/H,Y) &\simeq \Mor_H(G/L,Y),\nonumber\\
	\Morpmp_L(G/H,Y) &\simeq \Morpmp_H(G/L,Y),\\
	\Map_L(G/H,Z) &\simeq \Map_H(G/L,Z).\nonumber
\end{align}
\end{cor}

\begin{proof}
Indeed $\Map_L(G/H,Z)\simeq \Map_G(G/L\times G/H,Z)\simeq \Map_H(G/L,Z)$
using Corollary~\ref{p:main}. Similarly for the other maps.
\end{proof}
\begin{remark}\label{R:duality}
	Let us record the correspondence $\Map_L(G/H,Z)\simeq \Map_H(G/L,Z)$,
	explicitly.
	Denote $Y=G/L$ and $\sigma:Y\to G$ a Borel cross-section of the projection $g\mapsto gL$.
	Then the map $\ell:G\times Y\to G$ defined by
	\[
		\ell(g,y)=\sigma(g.y)^{-1}g\sigma(y)
	\]
	takes values in $L$, and forms a Borel cocycle.
	Given an $L$-map $f:G/H\to Z$ we define $F:G/L\to Z$ by
	\[
		F(y)=\sigma(y).f(\sigma(y)^{-1}H)
	\]
	and observe that for $h\in H$ one has using $L$-equivariance of $f$:
	\begin{eqnarray*}
		F(h.y)&=& \sigma(h.y).f(\sigma(h.y)^{-1}H)\\
		&=& \sigma(h.y).f(\ell(h,y)\sigma(y)^{-1}H)=\sigma(h.y)\ell(h,y).f(\sigma(y)^{-1}H)\\
		&=& \sigma(h.y)\ell(h,y)\sigma(y)^{-1}.F(y)=h.F(y).
	\end{eqnarray*}
	The correspondence $f\mapsto F$ is an explicit identification of $\Map_L(G/H,Z)$
	with $\Map_H(G/L,Z)$, depending on $\sigma:G/L\to G$.
\end{remark}
\medskip

For our purposes it will suffice to consider maps from measurable spaces
to Borel spaces. Hence we state the following Corollary for this case only,
others being analogous.

\begin{cor} \label{c:frob}
Let $H,L, M$ be closed subgroups of $G$ with $L<M$, and let $Z\in\BorelCat_G$.
Then the natural inclusion map
\[
	\Map_M(G/H,Z) \hookrightarrow \Map_L(G/H,Z)
\]
is a bijection if and only if the natural map
\[
	\Map_H(G/M,Z) \to \Map_H(G/L,Z),
\]
induced by $G/L\to G/M$, is a bijection.
%
% Similarly, the map
% \[
% 	\Map^0_M(G/B,Y) \hookrightarrow \Map^0_A(G/B,Y)
% \]
% is a bijection if and only if
% \[
% 	\Map^0_B(G/M,Y) \to \Map^0_B(G/A,Y)
% \]
% is a bijection.
\end{cor}

\begin{proof}
It is straightforward to check that one has the commutative diagram
\[\xymatrix{
	\Map_M(G/H,Z) \ar[d]^{\simeq} \ar[r] & \Map_L(G/H,Z) \ar[d]^{\simeq}\\
	\Map_H(G/M,Z) \ar[r] & \Map_H(G/L,Z)
}\]
where the vertical arrows are the bijections given by Corollary \ref{p:duality}.
The claim follows.
\end{proof}

\section{Borel density theorem}\label{s:borel_density} % (fold)

Fix a finite set $S$ and a collection of non-discrete local fields
$k_v$, $v\in S$, of zero-characteristic.
Let $V=\prod_{v\in S} \mathbf{V}_v(k_v)$ be an $S$-variety. It carries the
\emph{Zariski topology}, by which we mean the product of the Zariski
topologies of the factors $\mathbf{V}_v$, $v\in S$, and the
\emph{analytic topology} coming from the analytic
structures of the local fields.
%If it is not explicitly mentioned, we use the analytic topology.
An $S$-algebraic quotient $H\to K$ of an $S$-algebraic group $H$ is called {\it compact} if $K$ is compact
in the analytic topology.

For an $S$-algebraic variety $V$, we set
\[
	\Prob^0(V)=V,\qquad \Prob^{n+1}(V)=\Prob(\Prob^n(V))\qquad (n\ge 0).
\]
We shall need the following version of the Borel density theorem.
%We refer to \cite{Dani,Furst-density,Shalom-density, Wa1} for other
%generalizations
%of Borel's original theorem \cite{Borel-density}.

\begin{thm}[Borel Density] \label{th:BDT}
Let $H$ be an $S$-algebraic group with no compact $S$-algebraic factors,
and let $V$ be an $S$-algebraic $H$-space.
Then for all $n \geq 0$,
\[
	\Prob^n(V)^{H}=\Prob^n\left(V^{H}\right).
\]
\end{thm}
Note that existence of $H$-invariant measures implies existence of $H$-fixed points.
The set of $H$-fixed points $V^H$ is the product of $k_v$-points of fixed point varieties
\[
	V^H=\prod_{v\in S} \mathbf{V}_v^{\mathbf{H}_v}(k_v).
\]
\begin{proof}
The inclusion $\Prob^n(V)^{H}\supset\Prob^n\left(V^{H}\right)$ is trivial. We shall prove the other
inclusion by induction on $n$, starting with $n=1$.
When $V=\mathbf{V}_v(k_v)$ is an algebraic variety defined over a single local field $k_v$,
the claim is well-known (see, for instance, \cite[Theorem~3.9]{Shalom-density}).
To handle the general $S$-algebraic case, consider the projections
$p_v: V\to \mathbf{V}_v(k_v)$, $v\in S$.
For $\mu\in \Prob(V)^{H}$, the push-forward measures satisfy
\[
	p_{v*}(\mu)\in \Prob(\mathbf{V}_v(k_v))^{\mathbf{H}_v(k_v)}=\Prob(\mathbf{V}_v(k_v)^{\mathbf{H}_v(k_v)}),
\]
where $H_v = \mathbf{H}_v(k_v)$. This implies that $\mu$ is supported on
$V^H=\prod_{v\in S} V_v(k_v)^{H_v}$,
as required.

Assume validity of the Theorem for $n\geq 1$.
The barycenter map (see \cite{Phelps})
gives an $H$-equivariant map
\[
	{\rm bar}:\Prob(\Prob^{n}(V)) \to \Prob^n(V),
	\qquad \nu\mapsto \int_{\Prob^n(V)} \mu\, d\nu(\mu).
\]
By equivariance, invariant measures are mapped to fixed points.
Thus, by induction,
\begin{equation}\label{eq:bari}
	{\rm bar}(\Prob^{n+1}(V)^{H})\subset \Prob^{n}(V)^{H}=\Prob^n(V^{H}).
\end{equation}
If, for $\nu\in \Prob^{n+1}(V)^{H}$, we have
\[
	\nu(\{\mu:\, \mu(\Prob^{n-1}(V^{H}))<1\})>0,
\]
then
\[
	{\rm bar}(\nu)(\Prob^{n-1}(V^{H}))=\int_{\Prob^n(V)} \mu(\Prob^{n-1}(V^{H}))\, d\nu(\mu)<1,
\]
which contradicts (\ref{eq:bari}).
It follows that  $\nu$-a.e measure in $\Prob^{n}(V)$ is supported on $\Prob^{n-1}(V^{H})$.
Therefore, $\nu \in \Prob^{n+1}(V^{H})$. This completes the proof.
\end{proof}

We recall the notion of discompact radical introduced by Shalom in
\cite[Proposition~1.4]{Shalom-density}.
The \emph{discompact radical}
$\Rdc(H)$ of an $S$-algebraic group $H$ is the maximal $S$-algebraic subgroup of $H$ which
does not have any nontrivial compact $S$-algebraic quotients.
We note that $\Rdc(H)$ is normal in ${H}$ and $H/\Rdc(H)$ is compact
(see \cite{Shalom-density}).

\medskip

We give two corollaries of the Borel Density Theorem.

\begin{cor} \label{c:bdt}
Let $H$ be an $S$-algebraic group and $V$ an $S$-algebraic $H$-space.
Let $(X,\xi)$ be an $H$-space with an invariant probability measure.
Then for any $H$-map $\Phi:X\to \Prob^n(V)$,
\[
	\Phi(x) \in \Prob^n\left(V^{\Rdc(H)}\right) \ \ \mathrm{for} \ \xi
        \ \mathrm{a.e.} \ x.
\]
In particular, if the $H$-action on $(X,\xi)$ is ergodic, then for some $\eta_0\in\Prob^n(V)$
\[
	\Phi(x)\in H.\eta_0
\]
for $\xi$-a.e. $x\in X$.
\end{cor}

\begin{proof}
The first assertion follows from Theorem \ref{th:BDT} applied to the
measure $\Phi_*(\xi)$ and the group $\Rdc(H)$.
To prove the second assertion, observe that $\Phi_*(\xi)$ is ergodic with respect to the
action of the compact group $H/\Rdc(H)$. Hence it has to be supported on a single orbit.
\end{proof}

The following corollary explains the terminology ``Borel density theorem''.

\begin{cor} \label{c2:bdt}
Let $H$ be an $S$-algebraic group, and $M<H$
a closed subgroup, so that $H/M$ has a finite $H$-invariant measure.
Then $\Rdc(H)$ is contained in the Zariski closure of $M$.
\end{cor}

\begin{proof}
Let $\overline{M}$ be the Zariski closure of $M$, and $V=H/\overline{M}$.
Applying Corollary \ref{c:bdt} to the $H$-map $H/\overline{M}\to V$,
we deduce that $\Rdc(H)$ acts trivially on $H/\overline{M}$, hence
$\Rdc(H)< \overline{M}$.
\end{proof}

% section Borel (end)

\section{Proofs of the basic results} % (fold)
\label{sec:proofs_of_the_basic_results}

Following these preparations we are in position to prove Theorems
\ref{T:equiv-Borel-maps} and \ref{T:quasi-factors}. They correspond to the special cases $n=0,1$ of the following result.

\begin{thm}\label{T:algebraic-for-specialH}
	Let $G$ be an $S$-algebraic group, $H$ an $S$-algebraic subgroup
	with no compact $S$-algebraic factors, let $\Gamma<G$ be a lattice so that
	$\Gamma\acts G/H$ is ergodic, and let $V$ be an $S$-algebraic $G$-space.\\
	Then for every $n=0,1,\dots$, viewing $\Prob^n(V)$ as a Borel $\Gamma$-space,
	and $G/H$ as an ergodic $\Gamma$-space, there is a natural isomorphism
	\[
		\Map_\Gamma(G/H,\Prob^n(V))=\Map_G(G/H,\Prob^n(V))\cong \Prob^n(V^H)
	\]
	where $w\in \Prob^n(V^H)$ corresponds to the map $\Phi_w(gH)=g.w$.
\end{thm}
%
%Note that case $n=0$ is Theorem~\ref{T:equiv-Borel-maps}, and $n=1$ is the quasi-factors %Theorem~\ref{T:quasi-factors}.
%
\begin{proof}
Denote $Y=\Prob^n(V)$. By Corollary \ref{c:bdt}, we have the bijections
\[
	\Map_{H}(G/\Gamma,Y)\simeq Y^{H} \simeq \Map_{H}(G/G,Y).
\]
Hence, it follows from Corollary~\ref{c:frob} that the inclusion
\[
	\Map_{G}(G/H,Y)\hookrightarrow \Map_{\Gamma}(G/H,Y)
\]
is a bijection as well. This means that every $\Gamma$-map $G/H\to Y$ agrees
almost everywhere with a $G$-map.
Finally, it is easy to see that every $G$-map is of the given form.
\end{proof}

ֿ\begin{proof}[Proof of Theorem \ref{T:factors}] Let $\mu_y$ be the measures on $G/H$ coming from the disintegration of $\mu$. We may assume that $y \mapsto \mu_y$ is a $\Gamma$-map.
Applying Theorem \ref{T:quasi-factors} to the $\Gamma$-map $\phi: G/H \to \Prob(G/H)$ defined by $\phi(x) = \mu_{p(x)}$, we see that there is $\nu_0 \in \Prob(G/H)$, supported on $ (G/H)^H \cong \mathcal{N}_G(H)/H$, such that  $\phi(gH)=g\nu_0$. In particular $\phi$ is a $G$-map. Let $L$ denote the stabilizer of $\nu_0$. Since $y \mapsto \mu_y$ is injective, $Y$ is identified with its image under $y \mapsto \mu_y = \mu_{p(gH)} = g\nu_0$; that is $Y$ is isomorphic to $G/L$, and under this identification, $p$ is given by $gH \mapsto gL$.
Since $L < \mathcal{N}_G(H)$, we have that $H \triangleleft L$, and since the disintegration of $\mu_{G/H}$ has $L$-invariant probability measures as fiber measures, there is a finite invariant measure on $L/H$. This implies that $L/H$ is compact.
\end{proof}

ֿ\begin{proof}[Proof of Theorem \ref{T:joinings}] As explained in \S\ref{sub:relatively_p_m_p_joinings_and_quasi_factors}, a relatively p.m.p. joining of $G/H_1$ and $G/H_2$ gives rise to  a $\Gamma$-map $G/H_1 \to \Prob(G/H_2)$. By Theorem \ref{T:quasi-factors}, $(G/H_2)^{H_1}$ is nonempty, i.e. $H_1$ is contained in a conjugate of $H_2$. Reversing the roles of $H_1$ and $H_2$ we see that $H_1$ contains a conjugate of $H_2$. Since the $H_i$ are $S$-algebraic groups, this implies that $H_1$ and $H_2$ are conjugate.
\end{proof}

% \begin{proof}[Proof of Corollary~\ref{th:s-arth2}(2) and Corollary~\ref{c:qf}(2)]
% These are the special cases of Theorem \ref{t:algebraic} with $n=0$ and $n=1$.
% \end{proof}

% \combarak{I don't understand what the following is a proof of.}
%

% section proofs_of_the_basic_results (end)

%%%%%%%%%%%%%%%%%%%%%%%%%%%%%%%%%%%%%%%%%%%%%%%%%%%%%%%55

\section{The Mautner property}\label{s:mautner} % (fold)

Let $G$ be a locally compact second countable group and $H$ and $L$ be
closed subgroups.
We say that $(H,L,G)$ has the \emph{Mautner property} if
for every continuous unitary representation $\pi:G\to
\mathcal{U}(\Hilbert)$, a vector which is fixed by $H$ is already
fixed by $L$, i.e.
\[
	\Hilbert^{\pi(H)} \subset\Hilbert^{\pi(L)}.
\]
When $G$ is clear, we will say that the pair $(H,L)$ has the Mautner property.
In fact (fixing $G$ and $H$) there is a maximal $L$ such that $(H,L,G)$ has the Mautner property:
take the intersection of all fixators of all spaces of the form $\Hilbert^{\pi(H)}$
over all unitary $G$-representation $\pi$. % (note that there is no set theoretical problem here).
We call this maximal $L$ \emph{the Mautner envelope} of $H$ in $G$,
and denote it by $\MT{H}$. The alert reader will note that our
notation suppresses the dependence on $G$.
Thus $(H,\MT{H},G)$ has the Mautner property, and $(H,L,G)$ has the
Mautner property if and only if $L<\MT{H}$.
In this chapter we analyze Mautner envelopes in $S$-algebraic groups.
A closed subgroup of an $S$-algebraic group is called \emph{f.i.-algebraic}
if it is a finite-index subgroup of an $S$-algebraic subgroup $L\subset G$.
While the Mautner envelope does not have to be normal or f.i.-algebraic in general,
we shall show (see Theorem \ref{th:mau} below) that it always contains a cocompact
subgroup satisfying both of these properties.
Before doing so, let us collect some trivial observations.

\begin{lem} \label{l:mau}
Let $H$ be a closed subgroup in a locally compact second countable group $G$.
\begin{enumerate}
\item $\MT{\{e\}}=\{e\}$.
\item If $H$ is compact, then $H=\MT{H}$.
\item If $\MT{H}$ contains a closed subgroup $N$ which is normal in $G$, then
		\[\rhoֿ\left(\MT{H}\right)=\MT{\rho(H)}\]
		where $\rho:G\to G/N$ is the factor map.
\item If $N$ is a closed normal subgroup of $G$ containing $H$ then $\MT{H}<N$.
%\item If $H<N<G^0<G$ where $[G:G^0]<\infty$ and $N\triangleleft\, G^0$, then $\MT{H}<N$.
\end{enumerate}
\end{lem}

\begin{proof}\hfill{}\\
%\noindent{(1)}
%Since $\MT{\{e\}}$ is in the kernel of all unitary $G$-representations,
%it is trivial. Indeed, already the regular $G$-representation on $L^2(G)$ is faithful.\\
\noindent{(2)} Let $\pi: G \to H \backslash G$ be the quotient map. Given $g \notin H$,
let $f$ be a continuous compactly supported function on $H\backslash G$ which vanishes on $\pi(g)$ but not on $\pi(e)$. Then $f \circ \pi \in L^2(G)$ is $H$-invariant but not $g$-invariant.
% nonzero on $[H]$ and vanishes on $[Hg]$. Every compact group has a non-trivial fixed vector in $L^2(G)$
%whose stabilizer is compact.
%
%\\
%
%\noindent
(1) is a special case of (2). \\

\noindent{(3)}
The inclusion $\subset$ follows from the fact that any
$G/N$-representation is also a $G$-representation. For the other
inclusion, let $M = \rho^{-1}\left(\overline{\rho(H)}^{MT}
\right)$. We need to show that $(H,M)$ has the Mautner property. Since
$N < H$, given any unitary $G$-representation
$\pi:G\to\mathcal{U}(\Hilbert)$, one has
\[
	\Hilbert^{\pi(H)} \subset\Hilbert^{\pi(N)},
\]
and since $N$ is normal in $G$, $\Hilbert^{\pi(N)}$ is
$\pi(G)$-invariant, so induces a unitary $G/N$-representation
on $\Hilbert^{\pi(N)}$. In this representation
$\overline{\rho(H)}^{MT}$ acts trivially, so $M$ is trivial on
$\Hilbert^{\pi(N)}$. In particular $M$ fixes $\Hilbert^{\pi(H)}$.  \\

\noindent{(4)}
Can be seen by considering the representation $L^2(G/N)$.
%Consider the induced representation
%\[
%\Ind{G^0}{G}(L^2(G^0/N))\simeq \oplus_{G/G^0} L^2(G^0/N).
%\]
%Every vector in $L^2(G^0/N)\subset \Hilbert$ is $H$-invariant, but
%for every $g\in G\smallsetminus N$, there is a vector in $L^2(G^0/N)$ which is not $g$-invariant.
\end{proof}

\medskip

Let us review some basic properties of $S$-algebraic groups
that follow from well-known properties of algebraic groups over local fields of characteristic zero
(see, for instance, \cite{PR}).
%To emphasize the difference between $k$-algebraic groups and their $k$-points,
%we use boldface notation $\mathbf{G}$ to denote a product of $k_i$-groups
%$\prod_{v\in S} \mathbf{G}_v$, and $G$ for the locally compact second countable group
%given by the product $\prod_{v\in S} \mathbf{G}_v(k_v)$ of the
%$k_v$-points. Similarly for $S$-algebraic varieties.

\begin{prop}[cf. Proposition~3.3 and Theorem 6.14 in \cite{PR}]\label{p:alg1}\hfill{}\\
Let $\mathbf{X}= \prod_{v \in S} \mathbf{X}_v$ be a product of algebraic varieties over local fields with a transitive action of $\mathbf{G}= \prod_{v \in S} \mathbf{G}_v$.
Then the set $X$ of $S$-points of $\mathbf{X}$ is a union of finitely many $G$-orbits.
Each of these orbits is open and closed in $X$.
\end{prop}

The following is a special case of Proposition \ref{p:alg1}:

\begin{prop}\label{p:alg2}\hfill{}\\
Let $\rho:G\to H$ be an algebraic homomorphism of $S$-algebraic groups.
Then the image $\rho({G})$ is f.i.-algebraic in $H$.
% and the map ${\sf G}(S)/(\hbox{\rm Ker}(\phi)(S))\to \phi({\sf G}(S))$ is analytic isomorphism.
\end{prop}

\medskip

\begin{lem}\label{L:alm-alg}
	Let $\mathbf{G}$ be an $S$-algebraic group and $G$ the corresponding locally compact group.
	For any closed subgroup $M<G$ the collection $\mathcal{N}$ of all subgroups $N<M$ which are normal in $G$
	and correspond to Zariski connected, f.i.-algebraic subgroups in $G$,
	contains a unique maximal element.
\end{lem}

\begin{Def}\label{def:kernel}
	The unique maximal subgroup $N$ of $M$ as above is called the \emph{f.i.-algebraic kernel} of $M$,
	and is denoted by $\Na(M)$.
\end{Def}

To justify the term, we note that the f.i.-algebraic kernel $\Na(M)$ of $M$ is the
maximal f.i.-algebraic, Zariski connected subgroup of $G$
that acts trivially on $G/M$.

\medskip

\begin{proof}[Proof of Lemma~\ref{L:alm-alg}]\hfill{}\\
First let us show that a maximal element of $\mathcal{N}$ (with respect to inclusion)  must be unique.
Let $U, V \in \mathcal{N}$, and
% be two subgroups which are normal in $G$, Zariski connected, and almost algebraic.
%Denote
denote by $\mathbf{U}$, $\mathbf{V}$ their Zariski closures in $\mathbf{G}$,
and by $\overline{U}$, $\overline{V}$ the corresponding locally compact subgroups in $G$.
Then
\[
	[\overline{U}:U]<\infty,\qquad[\overline{V}:V]<\infty.
\]
%Let $\mathbf{G}^0=\prod_{\nu\in S} \mathbf{G}^0_\nu$ be the (product of $k_\nu$-) connected %component(s)
%of the identity in $\mathbf{G}=\prod_{\nu\in S} \mathbf{G}_\nu$.
%Since $\mathbf{U}$ and $\mathbf{V}$ are connected,
%$\overline{U}$, $\overline{V}$ are contained in the locally compact subgroup $G^0$
%corresponding to $\mathbf{G}^0$.
%Moreover, since $G^0$ is Zariski dense in $G$, we conclude
%that $\overline{U}$ and $\overline{V}$ are normal in $G^0$.
%Hence, $\overline{U}\cdot\overline{V}$ is a connected normal subgroup of $G^0$.
Applying Proposition \ref{p:alg1} to the action of $\overline{U} \times \overline{V}$ on $\overline{U\cdot V}$, we see that $\overline{U} \cdot \overline{V}$ is of finite index in $\overline{U\cdot V}$. Similarly $UV$ is of finite index in $\overline{U} \cdot \overline{V}$. Clearly $UV$ is normal and contained in $M$, therefore $UV \in \mathcal{N}$.
%It follows from Proposition \ref{p:alg1} that $\overline{U}\cdot \overline{V}$ consists of
%finitely many double cosets of $\overline{U}$ and $\overline{V}$ and, hence, of
%finitely many double cosets of $U$ and $V$.
%Since $U$ (and $V$) is normal,  this
%implies that $U\cdot V$ is a finite index subgroup of $\overline{U}\cdot \overline{V}$.
%Hence, the subgroup $U\cdot V$ is normal in $G$, Zariski connected,
%and almost algebraic.

We now show that a maximal element of $\mathcal{N}$ exists.
If $U_1< U_2<\cdots$ is an ascending chain in $\mathcal{N}$,
%of normal in $G$, Zariski connected,
%and almost algebraic subgroups of $G$,
then the corresponding chain of
Zariski closures %$\overline{U}_1< \overline{U}_2<\cdots$
stabilizes after
finitely many steps because the $\overline{U}_i$ are connected.
Since $U_i$ has finite index in $\overline{U}_i$, this implies that
the original chain stabilizes as well.
%This implies that there exists a unique maximal subgroup of $M$
%which is normal in $G$, Zariski connected, and almost algebraic.
\end{proof}

Below we use results
from \cite{MT,Wa2} in order to prove the following theorem,
which is the main result of this section. We remind our reader that in our setting all fields $k_v$ are of characteristic zero.

\begin{thm}[Mautner envelope] \label{th:mau}\hfill{}\\
Let $G$ be an $S$-algebraic group and $H$ an $S$-algebraic subgroup of $G$.
Then
\begin{enumerate}
	\item $\MT{H}/\Na\left(\MT{H}\right)$ is a compact group.
	\item $\MT{H}/\Na\left(\MT{H}\right)$ is a finite group, provided $H$ has no non-trivial
	compact $S$-algebraic quotients.
\end{enumerate}
\end{thm}

Let $\mathbf{H}$ be an algebraic group over a local field $k$, and $\mathbf{A}<\mathbf{H}$
be a one-dimensional connected algebraic subgroup, isomorphic over $k$ either to the
additive group $\mathbf{G}_a$ or to the multiplicative group $\mathbf{G}_m$.
We shall say that %the group of $k$-points $\mathbf{A}(k)<H=\mathbf{H}(k)$ is
$A < H$ is
a \emph{one-dimensional} $k$-\emph{split} subgroup of $H$. %$\mathbf{H}(k)$.
Given an $S$-algebraic group $\mathbf{H}$, let ${H}^\vee$ be the smallest closed subgroup
of the locally compact group $H=\prod_{\nu\in S} \mathbf{H}_\nu(k_\nu)$ containing
all the one-dimensional connected split subgroups over the relevant local fields.
Given an $S$-algebraic group ${G}$, containing $H$,
let ${H}^\wedge$ be the closed normal subgroup of $G$ generated by ${H}^\vee$
(the dependence on $G$ is hidden in this notation).

Theorem~\ref{th:mau} will be deduced from the following three propositions
that we prove below.

\begin{prop}\label{pr:mau}
Let ${G}$ be an $S$-algebraic group and ${H}$ an $S$-algebraic subgroup of $G$.
Then ${H}^\wedge$ is the Mautner envelope of ${H}^\vee$ in $G$.
\end{prop}

\begin{prop}\label{l:rdc}
Let $G$ be an $S$-algebraic group.
Then ${G}^\vee$ is a Zariski-dense subgroup of finite index in $\Rdc(G)$.
\end{prop}

\begin{prop} \label{pr:alg}
Let ${G}$ be an $S$-algebraic group and ${H}$ an $S$-algebraic subgroup of $G$.
Then ${H}^\wedge$ is f.i.-algebraic and Zariski connected.
\end{prop}

\begin{proof}[Proof of Theorem \ref{th:mau} (assuming Propositions \ref{pr:mau}, \ref{l:rdc}, \ref{pr:alg})]
	\hfill{}\\
Let $N={H}^\wedge$. By Proposition~\ref{pr:mau}, $N$ is contained in $\MT{H}$.
It is clearly normal, and also f.i.-algebraic and Zariski connected by Proposition~\ref{pr:alg}.
It follows that $N$ is contained in the f.i.-algebraic kernel $\Na\left(\MT{H}\right)$ of $\MT{H}$.
Hence,  in order to establish assertion (1) it suffices to show that $\MT{H}/N$ is compact.

Denote by $\rho:G\to G/N$ the factor map. By Lemma~\ref{l:mau}(3),
$\rho\left(\MT{H}\right)=\MT{\rho(H)}$, thus we need to establish the compactness of
%\[
$
	\MT{\rho(H)}.
$
%\]
Since $H^\vee<\ker(\rho)\cap H$ and $H/{H}^\vee$ is compact
by Proposition~\ref{l:rdc}, it follows that $\rho(H)$ is compact.
Now Lemma~\ref{l:mau}(2) implies that $\MT{\rho(H)}$ is compact, proving (1).

We now assume that $H$ has no compact $S$-algebraic quotients. % and, in particular, is connected.
By Proposition~\ref{l:rdc}, ${H}^\vee$ is Zariski dense in $H$.
Then the Zariski closure $\overline{N}$ of $N$ contains $H$.
By Proposition \ref{pr:alg}, $N$ is Zariski connected and of finite index in $\overline{N}$.
%, and $\mathbf{N}$ is connected (in the appropriate algebraic sense). \combarak{Is this what was previously %called Zariski-connected?}
%Let $\mathbf{G}^0$ be the connected component of $\mathbf{G}$.
%Since $G^0$ is Zariski dense in $\mathbf{G}^0$, it follows that
Clearly $\overline{N}$ is normal in $G$, so by
% Since $H<\overline{N}$, it follows from
Lemma \ref{l:mau}(4), $\MT{H}$ is contained in $\overline{N}$.
Hence, we have the inclusions
\[
	N< \Na\left(\MT{H}\right)<\MT{H}<\overline{N}
\]
that imply claim (2).
\end{proof}

\begin{proof}[Proof of Proposition \ref{pr:mau}]
Let  $A=\mathbf{A}(k)$ be a one-dimensional $k$-split algebraic subgroup of ${G}$.
We first show that $(A,{A}^\wedge)$ has the Mautner property.
Note that we may assume without loss of generality that $G$ is Zariski connected.
We need to consider the cases when $\mathbf{A}$ is isomorphic over $k$ to the additive group ${\mathbf{ G}}_a$ and
the multiplicative group $\mathbf{G}_m$.

When $\mathbf{A}\simeq \mathbf{G}_a$, it was proved in \cite[Proposition~2.1]{MT}
that there exists a closed normal subgroup $N$ containing $A$
such that $(A,N)$ has Mautner property. This implies the claim in this case.

When $\mathbf{ A}\simeq \mathbf{G}_m$, we use the results
of S.~Wang from \cite{Wa1}.
Let $g$ be an element of infinite order in $A$.
We define the following subgroups of $G$:
\begin{align*}
U^+&=\{x\in G:\, g^nxg^{-n}\to e\hbox{ as $n\to\infty$}\},\\
U^-&=\{x\in G:\, g^{-n}xg^{n}\to e\hbox{ as $n\to\infty$}\},\\
M&=\{x\in G:\, g^{-n}xg^{n}\hbox{ is bounded for $n\in\mathbb{Z}$}\}.
\end{align*}
It was shown in \cite[Section~2]{Wa1} that:
%there exist connected algebraic $k$-subgroups
%$\mathbf{U}^+$, $\mathbf{U}^-$, $\mathbf{M}$ in $\mathbf{G}$ such that:
\begin{enumerate}
	\item[(i)]
$U^+, \, U^-$ and $ M$ are the $k$-points of  Zariski connected $k$-algebraic subgroups.
\item[(ii)]
	${ U}^+$, ${ U}^-$ and ${ M}$ generate $G$.
	\item[(iii)]
	the closed subgroup $W=\overline{\langle U^+, U^-\rangle}$ is normal in $G$.
	\item[(iv)]
	The pair $(C,W)$ has the Mautner property, where $C=\overline{\langle g\rangle}$.
\end{enumerate}
We claim in addition, that
\begin{enumerate}
	\item[(iv)]
	$M$ commutes with $A$.
\end{enumerate}
For this we recall the construction of $M$ from \cite{Wa1}.
Let $\mathbf{S}$ be the maximal $k$-split torus in $\mathbf{G}$ containing $\mathbf{A}$,
and let $\Delta$ be a set of simple roots on $\mathbf{S}$ such that $|\alpha(g)|\le 1$ for $\alpha\in\Delta$,
where $|\cdot|$ denotes the absolute value of $k$.
Let $\Theta=\{\alpha\in\Delta:\, |\alpha(g)|=1\}$ and let $\mathbf{S}_\Theta$
be the connected component of the identity in $\bigcap_{\alpha\in\Theta} \ker(\alpha)$.
Then the subgroup $\mathbf{M}$ is precisely the centralizer of $\mathbf{S}_\Theta$ in $G$.
Every $\alpha\in \Theta$ defines a $k$-character of $\mathbf{A}$ such that $\alpha$
is bounded on $A=\mathbf{A}(k)$.
Hence, it follows that $\alpha(A)=1$ for every $\alpha\in \Theta$,
and $A\subset \mathbf{S}_\Theta(k)$. Therefore $M$ commutes with $A$, as claimed.

Let $N=\overline{AW},$ where $W$ is as in (iii).
Since $M$ commutes with $A$, it follows from (ii) that $N$ is normal in $G$.
By (iv), $(A,N)$ has the Mautner property.
Hence, $(A,{A}^\wedge)$ has the Mautner property as well. This proves our claim for the case $\mathbf{A} \simeq \mathbf{G}_m$.

%\marginpar{This last paragraph is the case of product of local fields, right? }
To finish the proof of the Proposition, we recall that $H^\vee$ is the closed subgroup generated by the one dimensional split subgroups $A$ of $H$, and observe that
$H^\wedge$ is the closed subgroup generated by the corresponding groups $A^\wedge$.
Since the pairs $(A,A^\wedge)$ have the Mautner property,
it follows that the pair $(H^\vee,H^\wedge)$ has the Mautner property as well.
As ${H}^\wedge$ is normal in $G$, we conclude that it is the Mautner envelope of $H^\vee$
by Lemma~\ref{l:mau}(4).
\end{proof}

Given an $S$-algebraic group $G$, we introduce its {\it Lie algebra} $\hbox{Lie}({ G})$.
It is defined to be $\prod_{v\in S} \hbox{Lie}(\mathbf{G}_v)$ where $\hbox{Lie}(\mathbf{G}_v)$
is the Lie algebra of the $v$th local factor.
Since each group $\mathbf{G}_v$ is defined over $k_v$,
its Lie algebras has a  $k_v$-structure, and we consider $\hbox{Lie}(\mathbf{G}_v)$ as a Lie algebra over $k_v$.
%We denote by $\hbox{Ad}:G_S\to \hbox{GL}(\hbox{Lie}(G))$ the adjoint representation.
The notion of the Lie algebra can be defined for any closed subgroup $H$
of $G=\prod_{v\in S} \mathbf{G}_v(k_v)$.
For this purpose we may assume without loss of generality that all the local fields in $S$
are incompatible (i.e., have different characteristics of the residue fields).
Then there exists an open normal subgroup of $H$ that splits
as a product of local factors (see e.g. \cite[Proposition~1.5]{r3}).
We define the Lie algebra of $H$ as the product of the Lie algebras of local factors.
The {\it exponential map} $\exp:\hbox{Lie}(H)\to H$ is the product of exponential maps
of the local factors,  and it defines a diffeomorphism in a neighborhood of zero.

\begin{proof}[Proof of Proposition~\ref{l:rdc}]
It suffices to prove the claim in case $\mathbf{G}$ is a connected algebraic group defined over a local
field $k$. Let $R=G^\vee$.
%It is a normal subgroup of $G$.
%Denote by $\bar{R}$ its Zariski closure.
%Then, by Theorem~\ref{th:alg}, $R$ is of finite index in $\bar{R}_k$, which is normal in $G_k$.
%We need to show that $R_{dc}(G)_k=\bar{R}_k$.
Consider the factor map $\pi: \mathbf{G}\to \mathbf{G}/\Rdc(\mathbf{G})$.
By Proposition \ref{p:alg2}, for a one-dimensional split group $A$, the group $\pi(A)$
is f.i.-algebraic and hence compact. Then $\ker(\pi)\cap A$ is infinite, and since $A$ is
one-dimensional and connected, $\ker(\pi)\cap A$ is Zariski dense in $A$. This implies that
$\pi(A)=1$. Hence, $R\subset \Rdc(G)$.

Since a discompact group has no finite index algebraic subgroups, it suffices to show that $R$ has finite index in $\Rdc(G)$.
We may assume, without loss of generality, that $G=\Rdc(G)$.
Since the field $k$ has characteristic zero, the group $\mathbf{G}$ has Levi decomposition
\[
	\mathbf{G}=\mathbf{S}\mathbf{T}\mathbf{U}
\]
where $\mathbf{S}$ is a connected semisimple subgroup defined over $k$, $\mathbf{T}$ is an algebraic torus
defined over $k$ that commutes with $\mathbf{S}$, and $\mathbf{U}$ is a normal unipotent subgroup defined over $k$.
Since $G=\Rdc(G)$, it follows from \cite[Theorem~3.1]{PR} that the reductive group $\mathbf{ST}$ has no
nontrivial $k$-anisotropic quotients.

In particular, this implies that the torus $\mathbf{T}$ is $k$-split.
Then $\mathbf{T}$ is isomorphic over $k$ to $\mathbf{G}_m^d$ and
$T\subset R$. Similarly, $U \subset R$; indeed, since $U$ is unipotent, the exponential map is a polynomial isomorphism $\hbox{Lie}(U)\to U$.
It follows that every element of $U$ is contained in one-dimensional
split unipotent subgroup of the form $\exp(tx)$, $x\in \hbox{Lie}(U)$.

Let $S^+$ denote the closed subgroup of $S$  generated by
%$A_k$ where
%$A$ is a
%one-dimensional
unipotent split subgroups of $S$.
Since $S$ does not have any $k$-anisotropic quotients,
the subgroup $S^+$ has finite index in $S$ by \cite[\S7.2]{PR}.
It follows that $R\cap STU$ has finite index in $STU$.
Since $G$ is a homogeneous space with respect to the action of $S\times T\times U$ given
by $(s,t,u)\cdot g=stgu^{-1}$, it follows from Proposition \ref{p:alg1} that $G$ is a finite union
of double cosets of $(ST, U)$. Using that $U$ is normal, we conclude that
$STU$ is a finite index subgroup of $G$. This implies the claim.
\end{proof}

Proposition \ref{pr:alg} will be deduced from the following more general theorem.

\begin{thm}\label{th:alg}
Let $\mathbf{G}$ be an algebraic group defined over a local field $k$ and $\mathbf{A}_i$, $i\in I$, a family of
connected $k$-subgroups closed under conjugation by elements of $G$.
Then the closed subgroup $N$ of $G$ generated by $A_i$, $i\in I$, is f.i.-algebraic.
\end{thm}

\begin{proof}[Proof of Proposition~\ref{pr:alg} (assuming Theorem \ref{th:alg})]
Note that $H^\wedge$ is generated by
all the conjugates of the %groups $A$,
%where $A=\mathbf{A}$ is a
one-dimensional split $S$-algebraic subgroups of $H$. By Theorem~\ref{th:alg}, it is f.i.-algebraic, and it is Zariski connected since it is generated by Zariski connected subgroups.
\end{proof}

%%%%%%%%%%%%%%%%%%%%%%%%%%%%%%% new version (uri) %%%%%%%%%%%%%%%%%%%%%%%%%5

The rest of this section will be devoted to the proof of Theorem \ref{th:alg}.
Our first step toward the proof of Theorem \ref{th:alg} is the case of a solvable group.

\begin{lem}\label{l:sol}
Let $\mathbf{G}$ be a solvable algebraic group defined over a local field $k$ and $\mathbf{A}_i$, $i\in I$, a family of
connected $k$-algebraic subgroups. Suppose there is no proper normal algebraic subgroup of $\mathbf{G}$ containing all of the $\mathbf{A}_i$. 
For each $i$ let $H_i$ be a finite index subgroup of $A_i$.
Let $N$ be a closed subgroup of $G$ 
which is normal in a finite index subgroup of $G$ and 
which contains all of the $H_i$.
Then $N$ is of finite index in $G$. 
% f.i.-algebraic.
\end{lem}

\begin{proof}
The group $\mathbf{G}$ must be connected, as it is generated by the connected groups $\mathbf{A}_i$ and their conjugates. The group $N$ is Zariski dense in $\mathbf{G}$ since it is normalized by a Zariski dense subgroup and its Zariski closure contains all of the $\mathbf{A}_i$. 
By dimension considerations we may assume that the collection $I$ is finite.
In case $\mathbf{G}$ is abelian the lemma follows from Proposition~\ref{p:alg2}, by considering the homomorphism
$\prod_I \mathbf{A}_i\to \mathbf{G}$. 

For the general case, we will proceed by induction on $\dim \mathbf{G}$. We will show that $\mathbf{G}$ contains a normal algebraic subgroup $\mathbf{M}$ such that $M \cap N$ is of finite index in $M$; from this the statement will follow by applying the induction hypothesis to $\mathbf{G}/\mathbf{M}$.  We will use the notation of \cite{Borel-book}, and let $\mathcal{D}^i(\mathbf{G})$ denote the subgroups in the derived series of $\mathbf{G}$. We will focus on the last nontrivial group $\mathbf{D}$ in this series, and distinguish two cases, according as $\mathbf{D}$ is or is not central in $\mathbf{G}$. 

If $\mathbf{D}$ is not central in $\mathbf{G}$, since $N$ is Zariski dense in $\mathbf{G}$, there is $n \in N$ which does not centralize $\mathbf{D}$. Consider the map $\varphi = \varphi_n: \mathbf{D} \to \mathbf{D}$ defined by $x \mapsto [n,x]=nxn^{-1}x^{-1}$. 
Using the fact that $\mathbf{D}$ is abelian we obtain the  formula
\begin{equation}\label{eq: formula}
%\[ 
[n,xy]=[n,x][n,y]^x=[n,x][n,y], 
%\]
\end{equation}
i.e., $\varphi$ is a $k$-algebraic homomorphism. Denote the Zariski closure of the image by $\mathbf{L}$. Applying Proposition \ref{p:alg2} we find that $\varphi(D)$ is of finite index in $L$. On the other hand, whenever $x \in N_G(N)$, $\varphi(x) \in N$. Since $N_G(N)$ is of finite index in $G$, we also have that $N_D(N)$ is of finite index in $D$. This implies that $L \cap N$ is of finite index in $L$, and hence for any $n \in N$, the conjugate $L^n$ also satisfies that $L^n \cap N$ is of finite index in $L^n$. 
Let $\mathbf{M}$ be the smallest $k$-algebraic group containing all the conjugates $\mathbf{L}^{n}$, for $n \in N$. Each $L^n$ is contained in $\mathbf{D}$ since $\mathbf{D}$ is normal in $\mathbf{G}$. Thus  $\mathbf{M}$ is contained in $\mathbf{D}$, hence abelian. Since $N$ is Zariski dense, $\mathbf{M}$ is a normal subgroup of $\mathbf{G}$. Since the theorem is already proved for abelian groups, we now apply it to the groups $\mathbf{M}$ (in place of $\mathbf{G}$) and $L^{n}$ (in place of $H_i$) to obtain that $M \cap N$ is of finite index in $M$, completing the proof in this case.

%then we let $L'$ be the group generated by $[n, D]$ for all $n \in N$, and let $\mathbf{L}$ denote the %Zariski-closure of $L'$. In view of \cite[Prop. 2.2]{Borel-book}, $L'=L$, the group of $k$-points of %$\mathbf{L}$. Since the normalizer $N_G(N)$ is of finite index in $G$, we also have $N_D(N)$ of finite index %in $D$. If $g \in N_G(N)$ and $n \in N$ then $[n,g] \in N$, so we obtain that $L\cap N$ is of finite index in %$L$. Since $N$ is Zariski-dense in $\mathbf{G}$, $D$ is not centralized by $N$, hence $\mathbf{L}$ is %non-trivial. Clearly $\mathbf{L}$ is normalized by $N$ and hence by $\mathbf{G}$. 

If $\mathbf{D}$ is central in $\mathbf{G}$, we let $\mathbf{E}$ be the preceding term in the derived  series, i.e. $\mathbf{D} = (\mathbf{E},\mathbf{E})$. For each $n \in N$ we can define $\varphi_n$ as in the previous case. Using \cite[Prop. 2.2]{Borel-book}, an induction shows that for each $i$, $N \cap \mathcal{D}^i(\mathbf{G})$ is Zariski-dense in $\mathcal{D}^i(\mathbf{G})$. In particular 
we can find $n \in N \cap E$ such that $\varphi_n(E)$ is non-trivial. We consider $\varphi = \varphi_n $ as a $k$-algebraic map $\mathbf{E} \to \mathbf{D}$. Since $\mathbf{D}$ is central in $\mathbf{G}$, (\ref{eq: formula}) shows that $\varphi$ is a homomorphism in this case as well. Now the same argument can be repeated. 
\end{proof}

\begin{lem}\label{l:ss}
Let $\mathbf{G}$ be a a connected $k$-algebraic group with no solvable non-trivial quotients.
Let $N$ be an open normal subgroup of $G$.
Then $N$ is of finite index in $G$.
\end{lem}

\begin{proof}
We denote the unipotent radical of $\mathbf{G}$ by $\mathbf{U}$, and let  $\mathbf{G}=\mathbf{S}\ltimes\mathbf{U}$ be a Levi decomposition.
The subgroup $\mathbf{S}$ is semi-simple (otherwise there is an abelian quotient).
By passing to a covering group, we may and will assume that $\mathbf{S}$ is simply connected.
Since $N$ is open in $G$,
it is enough to prove that $N$ is cocompact in $G$.
By \cite[Corollary 2.3.2(a)]{Margulis-book},
$N$ contains the group $S'$ where $\mathbf{S}'<\mathbf{S}$ is the subgroup
consisting of the product of all $k$-isotropic factors. Since $S'U$ is cocompact in $G$, there is no loss of generality in assuming that $\mathbf{S}'=\mathbf{S}$
and hence $S<N$. In case
$\mathbf{U}$ is trivial the proof is now complete.
For the general case, we proceed by induction on the dimension of $\mathbf{U}$.

Since $\mathbf{S}'$ is generated by unipotent subgroups and is not normal in $\mathbf{G}$
(otherwise there will be a solvable quotient),
there exists a connected unipotent subgroup $\mathbf{V}<\mathbf{S}$ which is not normalized by $\mathbf{U}$.
Since $U$ is Zariski dense in $\mathbf{U}$, we can find $u\in U$ which does not normalize $\mathbf{V}$.
Consider the group $\mathbf{W}$ generated by $\mathbf{V}$ and $\mathbf{V}^u$.
It is a subgroup of $\mathbf{V}\mathbf{U}$, hence unipotent.
$V$ and $V^u$ are contained in $N$, hence by Lemma~\ref{l:sol},
$N$ contains 
%an f.i.-algebraic subgroup of $\mathbf{W}$.
%It follows that $N$ contains 
a finite index subgroup of $W$.
Denote by $\mathbf{H}_1$ the connected component of the identity of $\mathbf{W}\cap\mathbf{U}$.
This is a connected unipotent subgroup of positive dimension in $\mathbf{U}$.
$N$ contains a finite index subgroup of $H_1$.
Consider all the conjugations of $H_1$ by elements of $N$.
These conjugates generate a subgroup $\mathbf{H}$ which is normal in $\mathbf{G}$.
$N$ contains a finite index subgroup of $\mathbf{H}$ by another application of Lemma~\ref{l:sol}.
The proof now follows by applying the induction hypothesis to the group $\mathbf{G}/\mathbf{H}$.
\end{proof}

\begin{proof}[{Proof of Theorem~\ref{th:alg}}]
Replacing $\mathbf{G}$ by the algebraic subgroups generated by the groups $\mathbf{A}_i$, we need to prove that $N$ is of finite index in $G$.
The group $N$ is clearly normal in $G$ and Zariski dense.
The group $\mathbf{G}$ must be Zariski connected, as it is generated by the connected groups $\mathbf{A}_i$.
By dimension considerations we can replace $I$ by a finite subset, so that $\mathbf{G}$ is still generated by the $\mathbf{A}_i$. 
By \cite[Proposition (2.2)]{Borel-book} and \cite[Theorem (2.5.3)(2)]{Margulis-book},
$N<G$ is open.
Let $\mathbf{G}_0$ be the smallest normal subgroup of $\mathbf{G}$ containing all the semisimple subgroups of $\mathbf{G}$, and let $\mathbf{G}_1 = \mathbf{G}/\mathbf{G}_0$. Then $N_0 = N \cap G_0$ is open in $G_0$, so of finite index by Lemma~\ref{l:ss}. By Lemma~\ref{l:sol}, the image of $N$ in $G_1$ is of finite index. This implies that $N$ is of finite index in $G$. 
%
%Then 
% 
%By , $N$ contains a finite index subgroup of the $k$-points of the normal subgroup generated by any %semo-simple subgroup of $\mathbf{G}$.
%We therefore may assume that $\mathbf{G}$ has no semi-simple subgroup.
%Then $\mathbf{G}$ is solvable, and the theorem follows by .
\end{proof}

%%%%%%%%%%%%%%%%%%%%%%%%%%%%5 old %%%%%%%%%%%%%%%%%%%%%%%%%%%%

\section{The relative Borel density theorem}\label{s:rel}

\subsection{Relative Borel density}
In this section we will state and prove the relative Borel density theorem (Theorem \ref{th:admiss}).
%Let $G$ be an $S$-algebraic group and $H$ an $S$-algebraic subgroup of $G$.
%Let $X,Y\in \MeasCat_G$ be measurable $G$-spaces.
%Our aim is to show that under suitable conditions every $H$-map $X\to Y$
%factors through a fixed homogeneous space $G/M$, where $M$ is a ``large" closed subgroup of $G$.
%More precisely

\begin{Def} \label{d:fat}
Let $G$ be an $S$-algebraic group and let $L$ be a closed subgroup.
A \emph{fat complement for  $L$ in  $G$}
is a subgroup $M<G$ satisfying the following three conditions:
\begin{itemize}
\item[(F1)] $L\cdot M=G$,
\item[(F2)] the Zariski closure of $M$ contains the discompact radical $\Rdc(G)$,
\item[(F3)] $\Na(M)\cap L$ is cocompact in $L$.
\end{itemize}
\end{Def}

\begin{thm}[Relative Borel density] \label{th:admiss}\hfill{}\\
Let $G$ be an $S$-algebraic group and $H$ an $S$-algebraic subgroup of $G$.
Let $(X,\xi)$ be a $G$-space with an invariant probability measure which is ergodic for the action of
$H$.
Then there exists a closed subgroup $M$ of $G$ which is cocompact, has finite covolume, and is a fat complement for $H$ in $G$, and a measure-preserving $G$-map $\pi:X\to G/M$, such that the following holds:

For every $S$-algebraic $G$-space $V$, for every $n \geq 0$,
and for every $H$-map $i:X\to \Prob^n(V)$, there exists an
$H$-map $j:G/M\to V$ such that $i=j\circ\pi$ almost everywhere.
That is, we have the following diagram:
\[
	\xymatrix{ X \ar[r]^{i} \ar[d]^\pi & \Prob^n(V) \\ G/M  \ar@{.>}[ur]^{j} & }
\]

\end{thm}
 Note that by property (F3), if $H$ has no compact algebraic factors, then $i$ is constant a.e.
We will see  below (Corollary \ref{cor: using fat complement}) that if $G$ has no compact $S$-algebraic factors then $i$ has an essentially  finite image.

\begin{proof}%[Proof of Theorem~\ref{th:admiss}]
Let $R=\Rdc(H)$. By Corollary \ref{c:bdt}, the image of $i$ is contained $\Prob^n\left(V^R\right)$. Let
$X \ec R$ denote the space of ergodic components for the action of $R$ on $(X, \xi)$, as described  in Appendix \ref{sec:measure_class_preserving_actions}.
Applying Proposition \ref{p:erg_comp}(1), we find that $i$ factors through $%\bar{i}:
X\ec R
%\to \Prob^n\left(V^R\right)
.$ Since $\xi$ is $G$-invariant and finite, Corollary \ref{c:ec} implies that $X \ec R$ is canonically identified with $X \ec \MT{R}.$ Let $N=\Na\left(\MT{R}\right)$ be the f.i.-algebraic kernel of the Mautner envelope of $R$, see Definition \ref{def:kernel}. It follows that $i$ factors through $X \ec N$ which, by Proposition~\ref{p:erg_comp}(2), is a $G$-space
on which $N$ acts trivially. Moreover the factor map $P_N:X \to X\ec N$ is a $G$-map.

Since $H$ acts ergodically on $X$, $H/H\cap N$ acts ergodically on $X\ec N$.
By Theorem~\ref{th:mau}(2), $N$ is of finite index in the Mautner envelope of $R$.
It follows that $N\cap R$ is of finite index in $R$.
Since $H/R$ is compact, $H/(H\cap N)$ is compact as well.
Because of ergodicity, the measure $(P_N)_*\xi$ is supported on a single orbit of $H$.
In particular, it follows that $X\ec N$ is isomorphic as a $G$-space
to $G/M$, where $M$ is a closed cocompact subgroup of $G$ of finite covolume, with $N<M$.

We denote by $\pi:X\to G/M$ the map corresponding to $P_N: X\to X\ec N$.
Then the existence of a Borel map $j:G/M\to Y$ such that $i=j\circ\pi$
follows from Proposition \ref{p:erg_comp}(1). Since $i$ and $\pi$ are $H$-maps, $j$ is an $H$-map as well.

It remains to show that $M$ is a fat complement for $H$ in $G$.
Since $H/(H\cap N)$ is compact and acts ergodically on $G/M$, we
conclude that $H$ acts transitively on $G/M$ and (F1) follows.
To prove (F2), we observe that $G/M$ supports a $G$-invariant measure $\pi(\xi)$.
Hence, it follows from Corollary~\ref{c2:bdt} that
$\Rdc(G)$ is contained in the Zariski closure of $M$.
Since $N$ acts trivially on $G/M$, we have $N<\Na(M)$.
This implies (F3), because $H\cap N$ is cocompact in $H$.
\end{proof}

%In particular, we deduce the description of the space of ergodic components:

%\begin{cor}
%Let the notation be as in Theorem \ref{th:admiss}.
%Then we have an isomorphism of $H$-spaces:
%\[
%	X\ec \Rdc(H)\simeq H/\overline{\Rdc(H)\cdot(M\cap H)}.
%\]
%Under this isomorphism the map $P_{\Rdc(H)}:X\to X\ec \Rdc(H)$ corresponds to the map
%\[
% 	X\stackrel{\pi}{\to} G/M \simeq  H/(M\cap H)\to H/\overline{\Rdc(H)\cdot(M\cap H)}.
%\]
%\end{cor}
%\begin{proof}
%Let $R=\Rdc(H)$ and and $N$ be the almost algebraic kernel of the Mautner envelope of $R$.
%So far we have obtained a $G$-map $\pi:X\to X\ec N\simeq G/M$.
%Applying Corollary \ref{c:ec} to the group $R$,
%we obtain that the isomorphism $X\ec R \to (X\ec N)\ec R$.
%Since
%\[
%	X\ec N \simeq G/M\simeq H/M\cap H
%\]
%as $H$-spaces, we conclude that
%\[
%	(X\ec N)\ec R \simeq (H/(M\cap H)) \ec R \simeq H/\overline{R\cdot(M\cap H)},
%\]
%as required.
%\end{proof}

%%%%%%%%%%%%%%%%%%%%%%%%%%%%%%%%%%%%%55

%Now we are in position to prove the relative Borel density theorem from the Introduction.

%\begin{proof}[Proof of Theorem \ref{th:admiss}]
%By Corollary~\ref{c:bdt}, the image of $i$ is almost everywhere contained in the subspace
%$\Prob^n(V_S^{R_{dc}(H)_S})$. This means that the map $i$
%is almost everywhere $R_{dc}(H)_S$-invariant.
%Hence, the claim follows from Theorem \ref{t:hw}.
%\end{proof}

%%%%%%%%%%%%%%%%%%%%%%%%%%%%%%%%%%%%%%%%%%%%%%%%%%%%%%

\subsection{Fat complements}

In this subsection we discuss fat complements. Our goal is the following proposition, which will provide more information on the conclusion of Theorem \ref{th:admiss}.

\begin{prop}\label{cor: using fat complement}
Let $G$, $H$, and $i$ be as in Theorem \ref{th:admiss}. If $G$ has no
compact $S$-algebraic factors then $i$ has finite image.

\end{prop}
Proposition \ref{cor: using fat complement} is an immediate consequence of the following.
\begin{thm} \label{t:M}
Let $ֿG$ be an $S$-algebraic group with no compact algebraic quotients.
Let $H$ be an $S$-algebraic subgroup, and let $M$ be a fat complement of $H$ in $G$.
Then $M$ is a  finite index subgroup of $G$.
\end{thm}

For a Lie algebra $\mathfrak{g}$, we denote by $\mathfrak{g}'$ its commutator subalgebra.
Let $\hbox{Ad}:G \to\hbox{GL}(\hbox{Lie}(G))$ denote the adjoint representation.
We recall that the differential of $\hbox{Ad}$ is given by $\hbox{ad}(x)=[x,\cdot]$, $x\in\hbox{Lie}(G)$,
and there exists a neighborhood $\mathcal{O}$ of the origin in $\hbox{Lie}(G)$ such that $\exp : \mathcal{O} \to G$ is well-defined and
\begin{equation}\label{eq:good}
\hbox{Ad}(\exp(x))=\exp(\hbox{ad}(x)) \quad \hbox{for all $x\in \mathcal{O}$.}
\end{equation}
For the proof of Theorem \ref{t:M} we will need the following version of the Malcev lemma, which  follows e.g.  from Corollary 7.9 in \cite{Borel-book}.
\begin{lem}%[cf. Corollary 7.9 in \cite{Borel-book}]
\label{l:malcev}
Let ${G}$ be an $S$-algebraic group, $\mathfrak{g}$ its Lie algebra, and $\mathfrak{m}$ a Lie subalgebra of $\mathfrak{g}$.
Suppose the subgroup $M$ of $G$ generated by $\exp(\mathfrak{m})$ is Zariski dense in $G$. Then
 $\mathfrak{g}'=\mathfrak{m}'$.
\end{lem}

%\comalex{We may state a relative version of this lemma:
%if Zariski closure of M contains L, then $Lie(L)'\subset \mathfrak{m}'$.
%If we need it.}

%\begin{proof}
%Let $\mathfrak{g}=\hbox{\rm Lie}({ G})$
%and denote by $M_0$ be a neighborhood of identity in $M$, which is the image under the exponential map %of a neighborhood $\mathcal{O}$ satisfying (\ref{eq:good}).
%Then it follows from (\ref{eq:good}) that
%$$
%(\hbox{Ad}(M_0)-id)\mathfrak{m}\subset \mathfrak{m}'.
%$$
%Then for $a_1,a_2\in \hbox{Ad}(M_0)$,
%\begin{align*}
%(a_1^{-1}a_2-id)\mathfrak{m} &\subset (a_1^{-1}-id)a_2 \mathfrak{m}+(a_2-id)\mathfrak{m}\\
%&\subset - a_1 (a_1^{-1}-id)\mathfrak{m} + \mathfrak{m}'\subset \mathfrak{m}'.
%\end{align*}
%Hence,
%$$
%(\hbox{Ad}(M)-id)\mathfrak{m}\subset \mathfrak{m}',
%$$
%and by Zariski density,
%$$
%(\hbox{Ad}({ G}_S)-id)\mathfrak{m}\subset \mathfrak{m}'.
%$$
%This implies that $[\mathfrak{g},\mathfrak{m}]\subset \mathfrak{m}'$. Then
%$$
%(\hbox{Ad}(M_0)-id)\mathfrak{g}\subset \mathfrak{m}'
%$$
%and as above
%$$
%(\hbox{Ad}(M)-id)\mathfrak{g}\subset \mathfrak{m}'.
%$$
%Using Zariski density again, we conclude that $\mathfrak{g}'\subset \mathfrak{m}'$, as required.
%\end{proof}

\begin{proof}[Proof of Theorem~\ref{t:M}]
Let $N=\Na(M)$. %\combarak{This used to be $K_a(M)$. I hope this fix was what was intended.}
Let $\mathfrak{g}$, $\mathfrak{h}$, $\mathfrak{n}$ and $\mathfrak{m}$ denote the Lie algebras corresponding to $G$, $H$, $N$ and $M$.
The hypothesis that $G$ has no compact factors, along with (F2), imply that $M$ is Zariski dense, hence
\begin{equation}\label{eq:norm}
\hbox{Ad}(G)\mathfrak{m}\subset \mathfrak{m}.
\end{equation}
We choose open subgroups $M^\circ$, $N^\circ$, and $G^\circ$ of $M$, $N$, and $G$
respectively such that
\begin{enumerate}
\item[(i)] $G^\circ$ normalizes $M^\circ$ and $N^\circ$, and $N^\circ$ normalizes $M^\circ$.
\item[(ii)] the group $M^\circ N^\circ$ is generated by open neighborhoods which are in the image of a neighborhood satisfying
 (\ref{eq:good}).
\end{enumerate}
We observe that one can choose these subgroups as products of local factors:
\[
 	M^\circ=\prod_{v} M_v^\circ,\qquad N^\circ=\prod_{v} N_v^\circ,\qquad
	G^\circ=\prod_{v} G_v^\circ
\]
(see \cite[Proposition~1.5]{ratnerduke}).  %BW updated reference without checking, hope this is ok.
For Archimedean $v$, (ii) holds provided that $M^\circ_v$ and $N^\circ_v$
are connected. For non-Archimedean $v$, the groups have bases of
neighborhoods of identity consisting of Lie subgroups, so that $M^\circ_v$ and $N^\circ_v$
can be taken to be sufficiently small to satisfy (ii). Property (i) can be satisfied
because of (\ref{eq:norm}).

Since $G$ is Zariski connected, the subgroup $G^\circ$ is Zariski dense in $G$.
Let $\bar{M}$ be the group of $S$-points of the Zariski closure of $M^\circ$. Then $\bar{M}$ is normal in $G$.
It follows from (F1)
and the Baire category theorem that the set $H{M}^\circ$ is open in $G$. In particular, it
is Zariski dense in $G$. Since the subgroup $H\bar{M}$ is both f.i.-algebraic
(by Proposition \ref{p:alg2}) and open in $G$,
we conclude that $H\bar{M}$ has finite index in $G$.

The group $N=\Na(M)$ is f.i.-algebraic and Zariski connected.
Since
$N^\circ$ is open in $N$,
the subgroup $N^\circ$ is Zariski dense in $\bar{N}$.
Let $L$ be the Zariski closure of the subgroup ${N}^\circ{M}^\circ$. Clearly, $L$ is normal and
$L\supset \bar{N}\bar{M}$. Hence,
$HL$ is of finite index in $G$.
Since $M$ is a fat complement of $H$, $H/(H\cap N)$ is compact.
Hence,
$G/L$ is compact as well.
By our assumption on $G$, the group ${N}^\circ{M}^\circ$
is Zariski dense in $G$. Hence, we may apply Lemma \ref{l:malcev}
with the subgroup  ${N}^\circ{M}^\circ$ to conclude that
\begin{equation}\label{eq:frak}
\mathfrak{g}'\subset \mathfrak{n}+\mathfrak{m}\subset \mathfrak{m}.
\end{equation}
Let $\mathbf{G}'$ denote the (algebraic) commutator subgroup of $\mathbf{G}$
and $\pi:\mathbf{G}\to \mathbf{G}/(\mathbf{G}'\bar{\mathbf{N}})$ the corresponding factor map.
Since the algebraic group $\mathbf{G}/(\mathbf{G}'\bar{\mathbf{N}})$ is abelian, it splits as an almost
direct product of anisotropic and split subgroups. Moreover, by our assumption on $G$,
the anisotropic component is trivial.
On the other hand, since $H/(H\cap N)$ is compact,  $\pi(H)$ is a compact f.i.-algebraic subgroup. Hence, $H\subset G'\bar{N}$.
Now using (F1) and (\ref{eq:frak}), we obtain
$$
\mathfrak{g}=\mathfrak{h}+\mathfrak{m}\subset \mathfrak{g}'+\mathfrak{n}+\mathfrak{m}=\mathfrak{m}.
$$
This shows that the group $M$ is open in $G$. On the other hand, the homogeneous space
$G/M$ is compact. Hence, $G/M$ has to be finite. This completes the proof.
\end{proof}

%%%%%%%%%%%%%%%%%%%%%%%%%%%%%%%%%%%%%%

\section{Completion of the proofs}\label{s:proof}
We will first state and prove a useful corollary of the results of the previous section.

\begin{cor} \label{c:mainthm}\hfill{}\\
Let $G$ be an $S$-algebraic group, $H$ an $S$-algebraic subgroup of $G$,
and $\Gamma$ be a lattice, such that $H\acts G/\Gamma$ is ergodic.

Then there exists a closed subgroup $\Gamma<M<G$,  where $M$ is cocompact, of finite covolume
in $G$, and a fat complement of $H$ in $G$, such that for every $S$-algebraic $G$-space $V$,
the inclusion map
\[
	\Map_M(G/H,\Prob^n(V)) \hookrightarrow \Map_{\Gamma}(G/H,\Prob^n(V))
\]
and the map
\[
	\Map_{H}(G/M,\Prob^n(V)) \to \Map_{H}(G/\Gamma,\Prob^n(V)),
\]
obtained by precomposing maps from $\Map_{H}(G/M,\Prob^n(V))$
with the projection $G/\Gamma \to G/M$, are bijections.
%Moreover, $M$ is a fat complement of $H$ in $G$ in the sense of Definition \ref{d:fat} below.
When $G$ has no nontrivial compact $S$-algebraic factors, $M$ is of finite index in $G$.

\end{cor}

\begin{proof}%[Proof of Corollary \ref{c:mainthm}]
Applying Theorem \ref{th:admiss} to $X=G/\Gamma$, we obtain a group $M$ and a $G$-map $G/\Gamma \to G/M$. Replacing $M$ with a conjugate we may assume that $\Gamma < M$.
The second bijection is a direct corollary of Theorem~\ref{th:admiss}
and the first follows formally from the second using Corollary~\ref{c:frob} using $L=\Gamma$. The last assertion follows from Proposition \ref{cor: using fat complement}.
\end{proof}

\subsection{Proof of Theorem \ref{T:general-quasi-factors} and Corollaries \ref{C:general-maps}, \ref{C:general-factors}}
\begin{proof}
Theorem \ref{T:general-quasi-factors} is just the first assertion of Corollary \ref{c:mainthm}, in the case $n=1$ (recall the bijection $X^{M\cap H} \cong \Map_H(G/M, X)$ described in the introduction). Corollary \ref{C:general-maps} is the case $n=0$, and the deduction of Corollary \ref{C:general-factors} from Theorem \ref{T:general-quasi-factors} follows the same steps as the deduction of Theorem \ref{T:joinings} from Theorem \ref{T:quasi-factors}.
\end{proof}

\subsection{Proof of Theorem~\ref{T:bfgw2}} % (fold)
\label{sub:proof_of_theorem}
\begin{proof}
Fix a map $f\in \Map_\Gamma(G/\Lambda,G/\Delta)$ where the target $G/\Delta$ is viewed as a Borel $\Gamma$-space.
By Corollary~\ref{p:duality} we have  the duality
\[
	\Map_\Gamma(G/\Lambda,G/\Delta)\cong \Map_\Lambda(G/\Gamma,G/\Delta).
\]
Given $f\in \Map_\Gamma(G/\Lambda,G/\Delta)$ let $F:G/\Gamma\to G/\Delta$ be the corresponding measurable $\Lambda$-equivariant map. As in Remark~\ref{R:duality} we fix a Borel cross-section $\sigma:Y=G/\Gamma\to G$
of the projection $G\to G/\Gamma$, and take
\begin{equation}\label{e:explicit}
	F(y)=\sigma(y).f\left(\sigma(y)^{-1}\Lambda\right).
\end{equation}
Let $\mu$ denote the probability measure on the space
\[
	X=Y\times Z,\qquad \textrm{where}\qquad Y=G/\Gamma,\quad Z=G/\Delta
\]
obtained by pushing the Haar measure $m_{G/\Gamma}$ to the graph of $F$.
Since $F$ is a $\Lambda$-map, $\mu$ is invariant under the action
of ${\rm Diag}(\Lambda)=\{ (\lambda,\lambda)\in G\times G :\lambda\in\Lambda\}$ on $X$.
Since $\Lambda$ is a Zariski dense subgroup of $G$, it acts ergodically on $G/\Gamma$. Hence so is the action
%As $\Gamma\acts (G/\Lambda,m_{G/\Lambda})$ is assumed to be ergodic, the action
%$\Lambda\acts (G/\Gamma,m_{G/\Gamma})$ is also ergodic, and hence so is
${\rm Diag}(\Lambda)\acts (X,\mu)$.
Next we want to use one of the assumptions (RS) or (BQ).

\medskip

\textbf{Assuming (RS)}, $\Lambda$ is a lattice in a subgroup of $G$
which is generated by unipotent elements. Note that we do not assume that this subgroup is connected; i.e.  $\Lambda$ itself may be generated by unipotent elements. In this situation one can apply the results of
Ratner \cite{r4}, and their extension by Shah \cite{Shah} and Witte \cite{Wi}, to deduce that $\mu$ is $L$-homogeneous,
where $L$ is a closed subgroup of $G\times G$ containing ${\rm Diag}(\Lambda)$.
This means that that there is $x_0=(g_1\Gamma,g_2\Delta)\in X$ so that denoting
\[
	\Sigma=L\cap (g_1\Gamma g_1^{-1}\times g_2\Delta g_2^{-1})
\]
is a lattice in $L$, and $\mu$ is the push-forward of the Haar measure $m_{L/\Sigma}$
to $L.x_0\subset X$.
Since $\mu$ projects onto $m_{G/\Gamma}$, it follows that $L$ projects onto $G$.
Consider
\[
	H=\{ g\in G :(e,g)\in L \}.
\]
This is a closed subgroup in $G$. Recall that $\mu$ is supported on the graph of $F:G/\Gamma\to G/\Delta$.
Hence for $m_{G/\Gamma}$-a.e. $y\in G/\Gamma$ and $h\in H$ we have $h.F(y)= F(y)$.
Note that given $z=g\Delta\in G/\Delta$ the group $\Delta_z=g\Delta g^{-1}$ is
defined unambiguously. We have
\[
	H\subset \Delta_{F(y)}
\]
for $m_{G/\Gamma}$-a.e. $y\in G/\Gamma$. Consider the measure $\eta=F_*m_{G/\Gamma}$ on $Z=G/\Delta$.
Since $H$ and $\Delta$ are closed sets, there is a conull (w.r.t. $\eta$) $Z_0 \subset Z$ such that
\begin{equation}\label{e:Deltay}
	H< \bigcap_{z\in Z_0} \Delta_z.
\end{equation}
Let $M$ denote the projection of $L<G\times G$ to the second factor.
Then $\Lambda<M$ and $\eta\in\Prob(Z)$ is $M$-homogeneous.
The connected component $M^0$ of the identity in $M$ is normal in $M$,
and therefore normalized by $\Lambda$, which is Zariski dense in $G$.
Hence $M^0$ is normal in $G$.
Since $G$ is simple and connected, we have
\begin{enumerate}
	\item either $M^0=\{e\}$ and $M$ is discrete,
	\item or $M^0=M=G$.
\end{enumerate}
Case (1). As $M$ is discrete, $\eta$ is atomic, and since it is a $\Lambda$-invariant and ergodic probability measure, it is supported on a finite $\Lambda$-orbit. That is, $F$ has a finite image.  Let $\Lambda_0$ be a finite index normal subgroup of $\Lambda$ acting trivially on the image of $F$. By the Howe-Moore theorem, $\Lambda_0$ also acts ergodically on $G/\Gamma$, which implies that $F$ is essentially constant. That is,
%Ergodicity of $\Lambda\acts (G/\Delta,\eta)$ then implies
$\eta=\delta_{z_0}$ for some
$\Lambda$-fixed point $z_0=a\Delta\in G/\Delta$.
%\combarak{I do not understand this point, maybe $\eta$ is supported on a finite orbit. }
We deduce $\Lambda\subset \Delta_{z_0}=a\Delta a^{-1}$ from (\ref{e:Deltay}),
and have $F(g\Gamma)=y_0$ a.e. It follows from (\ref{e:explicit}) that
\[
	f(g\Lambda)=g a\Delta.
\]
Case (2): $M=G$ and $\eta=m_{G/\Delta}$. Since $Z_0$ in (\ref{e:Deltay}) is conull, it follows that $H=\{e\}$.
%\[
%	H<\bigcap_{z\in Z_0}\Delta_z<\bigcap_{g\in G} \Delta_{g}=\{e\}.
%\]
Recalling the definition of $H$, we deduce that
\[
	L=\{ (g,\rho(g)) : g\in G \}
\]
for some continuous homomorphism $\rho:G\to G$, which is therefore algebraic.
As ${\rm Diag}(\Lambda)<L$ we get $\rho(\lambda)=\lambda$ for $\lambda\in\Lambda$;
and Zariski density of $\Lambda$ in $G$ implies $\rho(g)=g$ for all $g\in G$.
Therefore $F(gg_1\Gamma)=gg_2\Delta$, or
\[
	F(g\Gamma)=gg_0\Delta
\]
with $a=g_1^{-1}g_2$; in particular $\Gamma<\Delta_{a}$.
From (\ref{e:explicit}) we deduce, using a Borel cross-section $s:G/\Lambda\to G$, that
\[
	f(g\Lambda)=s(g\Lambda) F\left(s(g\Lambda)^{-1}\Gamma\right )=s(g\Lambda)s(g\Lambda)^{-1}a\Delta=a\Delta,
\]
which is an a.e. constant function.
This completes the proof of Theorem~\ref{T:bfgw2} under assumption (RS).
Let us proceed to the proof under assumption (BQ).

\medskip

\textbf{Assuming (BQ)}, we have $\Delta<\Delta_0$ where $\Delta_0$ is a lattice in $G$.
Denote by $\pi:G/\Delta\to G/\Delta_0$ the natural projection, and let
\[
	F_0:G/\Gamma\overto{F} G/\Delta\overto{\pi} G/\Delta_0,\qquad \eta_0=\pi_*\eta\in\Prob(G/\Delta_0).
\]
Then $\eta_0$ is a $\Lambda$-invariant and ergodic probability measure on $G/\Delta_0$.
Since $\Lambda$ is assumed to be Zariski dense in $G$, we can apply the recent result
of Benoist-Quint \cite{BQ} to the action $\Lambda\acts G/\Delta_0$ to deduce the dichotomy:
\begin{enumerate}
	\item either $\eta_0$ is atomic, equidistributed on a finite $\Lambda$-orbit $\Lambda g_0\Delta_0\subset G/\Delta_0$,
	\item or $\eta_0=m_{G/\Delta_0}$ is the Haar measure on $G/\Delta_0$.
\end{enumerate}
In case (1), $\eta$ is also atomic, and we conclude the proof as in the previous case.
%it follows that $\eta$ on $G/\Delta$ is atomic.
%Since it is a $\Lambda$-invariant probability measure, this forces $\eta$ to be the Dirac measure %$\delta_{z_0}$ at some $\Lambda$-fixed point $z_0\in G/\Delta$.
%As in the (RS) case the argument is completed by showing that the $\Lambda$-map
%$F:G/\Gamma\to G/\Delta$ is constant iff the $\Gamma$-map $f:G/\Lambda\to G/\Delta$ is a $G$-map.
We are left with case (2), where the probability measure $\eta$ on $G/\Delta$ projects
onto the normalized Haar measure $m_{G/\Delta_0}$ on $G/\Delta_0$.
We claim that this is possible, only if $\Delta$ has finite index in $\Delta_0$, and $\eta$
is the normalized Haar measure $m_{G/\Delta}$.

First let us identify the $G$-action on $G/\Delta$ with the skew-product $G$-action on
$G/\Delta_0\times \Delta_0/\Delta$ given by
\[
	g_1: (g\Delta_0,a\Delta)\mapsto (g_1g\Delta_0, c(g_1,g\Delta_0)a\Delta),
\]
where $c:G\times G/\Delta_0\to\Delta_0$ is the cocycle
\[
	c(g_1,g\Delta_0)=\sigma(g_1g\Delta_0)^{-1}g_1\sigma(g\Delta_0)
\]
associated to a choice of a Borel section $\sigma:G/\Delta_0\to G$ for the projection $G\to G/\Delta_0$.
That is we have a Borel isomorphism
of the $G$-actions on
\[
	G/\Delta\quad \cong\quad G/\Delta_0\times \Delta_0/\Delta.
\]
Consider the restriction to the action of $\Lambda<G$, and view
the $\Lambda$-invariant and ergodic probability measure $\eta$ realized on
$G/\Delta_0\times \Delta_0/\Delta$.
Since $\eta$ projects to the Haar measure $m_{G/\Delta_0}$, the disintegration of $\eta$
with respect to $m_{G/\Delta_0}$ has the form
\[
	\eta=\int_{G/\Delta_0} \eta_{x}\,dm_{G/\Delta_0}(x),
	\qquad
	\eta_x \in \Prob(\Delta_0/\Delta).
\]
Moreover, for $\lambda\in\Lambda$ and $m_{G/\Delta_0}$-a.e. $x$ one has
\begin{equation}\label{e:c-equiv}
	 \eta_{\lambda.x}=c(\lambda,x)_*\eta_x
\end{equation}
because $\eta$ is $\Lambda$-invariant. The set $J=\Delta_0/\Delta$ is at most countable,
so each probability measure $\eta_x$ on $J$ is atomic and therefore has a well-defined
maximal `weight'. That is, for $x\in G/\Delta_0$ we define
\[
	p(x)=\max_{j\in J}\eta_x(\{j\}),\qquad A(x)=\left\{ i\in J  :  \eta_x\left(\{i\}\right)=p(x)\right\}.
\]
It follows from (\ref{e:c-equiv}) that
\[
	A(\lambda.x)=c(\lambda,x)^{-1} A(x)\qquad (x\in G/\Delta_0,\ \lambda\in\Lambda)
\]
where $A(x)\subset J=\Delta_0/\Delta$ are finite sets; the cardinality $|A(x)|$ being
a.e. constant by ergodicity of $\Lambda\acts (G/\Delta_0,m_{G/\Delta_0})$.

We claim that $J$ is finite. Assume otherwise, and consider the probability space
\[
	(Y,\nu)=(Y_0,\nu_0)^J
\]
where $(Y_0,\nu_0)$ is a non-trivial probability space, say $\{0,1\}$ with $(1/2,1/2)$-weights.
Since $\Delta_0\acts J=\Delta_0/\Delta$ is a transitive action on an infinite index set,
the corresponding $\Delta_0$-action on $(Y,\nu)$ is an ergodic measure-preserving action.
Consider the induced $G$-action
\[
	G\acts (G/\Delta_0\times Y,m_{G/\Delta_0}\times \nu),\qquad g:(x,y)\mapsto (g.x, c(g,x).y).
\]
Actions induced to $G$ from ergodic actions of a lattice $\Delta_0<G$
are ergodic. So $G\acts G/\Delta_0\times Y$ is ergodic.
By Moore's theorem, the restriction to any unbounded subgroup, in particular
to $\Lambda$, remains ergodic.
Recall the sets $A(x)\subset J$ and the fact that $Y=\{0,1\}^J$.
Consider the set
\[
	Z=\{(x,y)\in G/\Delta_0\times \{0,1\}^J : \forall j\in A(x),\ y(j)=0\}.
\]
Observe that $m_{G/\Delta_0}\times \nu(Z)=2^{-k}>0$ where $k=|A(x)|$,
and that $Z$ is invariant under the $\Lambda$-action.
This contradicts ergodicity of the $\Lambda$-action,
showing that the assumption that $J$ is infinite was wrong.

We have proved that $[\Delta_0:\Delta]<+\infty$; in particular $\Delta<G$
is a lattice. Hence \cite{BQ} implies that the $\Lambda$-invariant and ergodic
probability measure $\eta$ on $G/\Delta$ is the normalized Haar measure $\eta=m_{G/\Delta}$
(because it cannot be atomic under our assumption).
Now consider the pushforward measure $m$ of $m_{G/\Gamma}$ to the
graph of $F:G/\Gamma\to G/\Delta$.
This is a probability measure on
\[
	G\times G/\Gamma\times\Delta
\]
invariant under the diagonal action of $\Lambda$.
By \cite{BQ} such a measure should be homogeneous for a subgroup $L<G\times G$
containing $\Lambda$.
The argument can now be completed as in the (RS) case, by ruling out
the possibility that $L=G\times G$ (because $m$ cannot be a product measure
$m_{G/\Gamma}\times m_{G/\Delta}$), and deducing that the $\Lambda$-equivariant
map $F:G/\Gamma\to G/\Delta$
is actually $G$-equivariant. This, in turn, implies that the
original $\Gamma$-equivariant map $f:G/\Lambda\to G/\Delta$ is a constant map.
\end{proof}

In order to prove Theorem~\ref{cor: non-lattices1} we use recurrence instead of
the existence of a finite invariant measure.
An action of a group $A$ on an $A$-space $Z$ is called {\bf strictly conservative} if for
any non-null Borel $B \subset Z$ and any compact $C \subset
A$, there is $a \in A \sm C$ such that $a B
\cap B$ is non-null.

\begin{lem}\label{l:rec}
Let $A$ be an lcsc group and $Y$ a second countable topological space
equipped with a Borel probability measure $\nu$ on which $A$ acts strictly conservatively.
Then for almost every $y\in Y$, there is a sequence $a_n\in A$ such that
$a_n\to \infty$ and $a_n y \to y$.
\end{lem}

\begin{proof}
For an open $U \subset Y$ of positive measure and compact $C \subset A$,
let
$$Y_{C, U} = \{ z\in {U} : \forall a
\in A \sm C,
\, a z \notin {U}  \}.$$
This is a nullset by strict conservativity.
Taking an
exhaustion of $A$ by countably many compacts $\{C_i\}$ and a
countable basis ${U}_j$ of open sets of $Z$, we have that
$Y_\infty = \bigcup_{i,j:\nu(U_j)>0} Y_{C_i,
{U}_j}$ is also a nullset, and any point in
$\hbox{supp}\, \nu \sm Y_\infty$ has the required properties.
\end{proof}

\begin{proof}[Proof of Theorem \ref{cor: non-lattices1}]
We identify the group $G$ of orientation preserving isometries of $\HH$ with
$\hbox{PSL}_2(\mathbb{R})$. Without loss of generality we may assume that $\Gamma\subset G$.
The space $\partial \HH \times \partial \HH$ is isomorphic
as a $G$-space to the factor space $G/A$ where $A$ is the diagonal subgroup of $G$.
By Corollary \ref{p:duality},
\begin{align}\label{eq:mapp}
	\Map_\Gamma(G/A,G/A) \simeq \Map_A(G/\Gamma,G/A).
\end{align}
The  action of $A$ on $G/\Gamma$ is precisely the geodesic flow on the unit tangent bundle of the hyperbolic
surface $\HH/\Gamma$. According to \cite{Hopf}, the geodesic flow is either ergodic and
conservative, or totally dissipative.
Since $\Gamma$ acts ergodically on $G/A$, we are dealing with the former case.

Let $\Phi:G/\Gamma\to G/A$ be an $A$-map.
Let $\mu$ be a Haar measure on $G/\Gamma$ and $\nu=\Phi_*\mu$.
Then $\nu$ is ergodic and conservative with respect to the action of $A$.
From this we deduce that  for $\nu$-almost every $x\in G/A$ there exists a sequence
$a_i\to \infty$ in $A$ such that $a_ix\to x$ (Lemma~\ref{l:rec}).
On the other hand, it is easy to check that the sequence $a_ix$ may only accumulate
on the set $(G/A)^A$ of fixed points of $A$. This shows that for $\mu$-almost every $x\in G/\Gamma$,
we have $\Phi(x)\in (G/A)^A.$ Note that $(G/A)^A\simeq \mathcal{N}_G(A)/A$ where $N_G(A)$ denotes the normalizer
of $A$ in $G$. Now it follows from ergodicity that the set $\hbox{Map}_A(G/\Gamma,G/A)$ consists of exactly
two elements indexed by $N_G(A)/A$. Under identification \eqref{eq:mapp},
these two elements correspond to the maps
\[
	G/A\to G/A,\qquad gA\mapsto gnA, \qquad (n\in \mathcal{N}_G(A)/A),
\]
which are the identity map and the flip.
\end{proof}

%%%%%%%%%%%%%%%%%%%%%%%%%%%%%%%%%%%%%%%%%%%%%%%%%

\section{Examples}\label{s:ex}

In this section we give several elementary examples which
illustrate that the assumptions imposed in our main results are essential.

\begin{example} \label{ex:product}
{\rm
This example demonstrates that a fat complement $M$ may be of infinite index, and that the subgroup $M$ in Theorem \ref{C:general-maps} may in general be of  infinite index. Consider an $S$-algebraic group $\mathbf{G}=\mathbf{H}\times \mathbf{L}$
such that $H$ is compact and set $\Gamma=L$, $Y=G$.
Then
$$
\Phi:G/H\to Y,\ \ lH\mapsto (e,l),\quad l\in L,
$$
defines a $\Gamma$-map. The largest group which could play the role of $M$ in the statement is $M=L$.

}
\end{example}

\begin{example}{\rm
This example shows that the group $M$ in Theorem \ref{T:general-quasi-factors} does not have to be f.i.-algebraic.
Let $\mathbf{G}=\mathbf{\hbox{SO}_n}\times \mathbf{\hbox{G}_a}$ where $\mathbf{\hbox{SO}_n}$ denotes the orthogonal group and
$\mathbf{\hbox{G}_a}$ denotes the additive group. We fix a compact one-parameter subgroup
$\{k(t)\}_{t\in\mathbb{R}}$ of $\hbox{SO}_n(\mathbb{R})$ and set $\Gamma=\{(k(t),t)\}_{t\in\mathbb{R}}$.
We also set $H=\hbox{SO}_n$. Then the map
$$
\Phi:G/H\to G, \ \  (e,t)H\mapsto (k(t),t),\quad t\in \mathbb{R},
$$
is $\Gamma$-equivariant, but not equivariant almost everywhere with respect to any larger subgroup.
}
\end{example}

\begin{example} {\rm
This example shows that even when $G$ in Theorem \ref{T:general-quasi-factors} has no
nontrivial compact $S$-algebraic quotients, the subgroup $M$ could be proper (cf. Theorem \ref{t:M}).
Let $G=\hbox{G}_m$, $H=\left<\pm 1\right>$, $\Gamma=\mathbb{R}_+^\times$,
and
$$
\Phi:\mathbb{R}^\times/\left<\pm 1\right> \to \mathbb{R}^\times, \ \ x\left<\pm 1\right>\mapsto x, \quad
x\in \mathbb{R}_+^\times,
$$
be the factor map. Then  $\Phi$ is $\mathbb{R}_+^\times$--equivariant but not
$\mathbb{R}^\times$-equivariant.
}
\end{example}

\begin{example} \label{ex:factor}
{\rm
This example shows that the assumption in Corollary \ref{C:general-factors}
%Theorem \ref{T:general-quasi-factors}
that the factor is measure-preserving
is essential. Let $\mathbf{G}=\textrm{PSL}_2$, $k=\mathbb{R}$, and let $H$ be the diagonal subgroup of $G$.
Let $\Gamma$ be a cocompact lattice in $G$ and $X=G/H$ equipped with the Haar measure class $\mu$.
Then  $\Gamma$ acts ergodically on $X$, and
Theorem \ref{T:general-quasi-factors}
implies that every measure-preserving $\Gamma$-factor of $X$ is of the form
$G/Q$ where  $Q$ is a closed subgroup of $G$ such that $H$ is a normal cocompact subgroup of $Q$.
Hence, there are only two measure-preserving $\Gamma$-factors of $X$
corresponding to $Q=H$ and $Q=\mathcal{N}_{G}(H)$.

On the other hand, there exists an infinite normal subgroup $\Lambda$ of $\Gamma$
such that $\Lambda$ does not act ergodically on $X$. For instance,  one can take
$\Lambda$ such that $\Gamma/\Lambda\simeq \mathbb{Z}^k$ with $k\ge 3$ (see \cite{Rees}).
Then there is a nontrivial measure-class-preserving $\Gamma$-factor $Y=(X\ec \Lambda, P_\Lambda(\mu))$
(see Appendix \ref{sec:erg}). Since $\Lambda$ acts trivially on $Y$ and $G$ is simple,
the space $Y$ cannot be isomorphic to a $G$-space.
}
\end{example}

\begin{appendix}
	
%\section{Measure class preserving actions} % (fold)

\section{The space of ergodic components}\label{sec:erg}
\label{sec:measure_class_preserving_actions}

Let $G$ be an lcsc group. Given a $G$-space $(X,\mu)$, we denote by $\mathcal{B}(X)$ the Boolean
$\sigma$-algebra consisting of measurable sets modulo the ideal of null sets.
Note that every measure-class-preserving $G$-map $\Phi:(X,\mu)\to (Y,\nu)$
between $G$-spaces
induces a $G$-equivariant Boolean-$\sigma$-algebra homomorphism $\Phi^*:\mathcal{B}(Y)\to\mathcal{B}(X)$
defined by $\Phi^*([A])=[\Phi^{-1}(A)]$ for a Borel subset $A\subset X$.

For a $G$-space $(X,\mu)$, we introduce the space of ergodic components $X\ec G$.
This space can be characterized by the universal property ---
Proposition \ref{p:erg_comp}(1) below.

\begin{prop}[Ergodic decomposition]\label{p:erg}
Given a $G$-space $(X,\mu)$, there exist a standard Borel space $X\ec G$
and a Borel map $P_G:X\to X\ec G$ such that
\begin{itemize}
\item $P_G(gx)=P_G(x)$ for all $g\in G$ and almost every $x\in X$,
\item $P_G^*(\mathcal{B}(X\ec G))$ is the algebra of $G$-invariant elements in $\mathcal{B}(X)$.
\end{itemize}
\end{prop}

Proposition \ref{p:erg} is a part of folklore in ergodic theory,
see, for instance, \cite[Theorem~5.2]{G+S} for a proof.

We shall use the following functorial properties of the space $X\ec G$:

\begin{prop}\label{p:erg_comp}
\begin{enumerate}
\item Let $(X,\mu)$ be a $G$-space, $Y$ a standard Borel space, and $\Phi:X\to Y$ a Borel map
such that $\Phi(gx)=\Phi(x)$ for all $g\in G$ and almost every $x\in X$.
Then there exists a Borel map $\bar\Phi:X\ec G\to Y$ such that $\Phi=\bar\Phi \circ P_G$ on a conull set.
%That is, we have the following diagram:
%\begin{diagram}
 %& \rTo^\Phi & Y\\
 %\dTo^{\tiny P_G} &\ruDashto_{\tiny\bar \Phi} & \\
 %X\ec G & &
 %\end{diagram}

%\item Let $G$ be an lcsc group, $H_i$, $i\in I$, be a family of closed subgroups of $G$
%and $H$ the closed subgroup generated by $H_i$'s. Let $(X,\mu)$ be a $G$-space,
%$Y$ a standard Borel space, and $\Phi:X\to Y$ a Borel map such that
%$\Phi(hx)=\Phi(x)$ for all $h\in \cup_{i\in I} H_i$ and a.e. $x\in X$.
%Then $\Phi(hx)=\Phi(x)$ for all $h\in H$ and a.e. $x\in X$.
%In particular, $\Phi$ factors through $P_H$.

\item If $H$ is a closed normal subgroup of an lcsc group $G$ and $(X,\mu)$ is a $G$-space,
then $X\ec H$ is equipped with a structure of a $G$-space so that
$P_H:X\to X\ec H$ is a $G$-map.
\end{enumerate}
\end{prop}

\begin{proof}
In the proof we shall use the well-known correspondence between standard Borel spaces
and Boolean algebras (see \cite{mac, ram}). Specifically, we shall use
the following:
\begin{enumerate}
\item[(i)] (\cite[Theorem 2.1]{ram}) If $\Phi,\Psi:(X,\mu)\to (Y,\nu)$ are $G$-maps
such that $\Phi^*=\Psi^*$, then $\Phi=\Psi$ on a conull set.
\item[(ii)] (\cite[Theorem 3]{mac}, \cite[Theorem 3.6]{ram})
If $(X,\mu)$ and $(Y,\nu)$ are $G$-spaces and $\phi:\mathcal{B}(Y)\to \mathcal{B}(X)$
is a $G$-equivariant homomorphism, then there exists a measure-class-preserving $G$-map
such that $\Phi^*=\phi$.
\item[(iii)] (\cite[Theorem 1]{mac}, \cite[Theorem 3.3]{ram}) Every Boolean space
equipped with a $G$-action is $G$-equivariantly isomorphic to a Boolean measure space associated with
a $G$-space.
\end{enumerate}

Now we start with the proof of (1).
By \cite[B.5]{Zimmer-book} there exists a measurable map $\Psi:X\to Y$ such that
$\Psi=\Phi$ on a conull set and $\Psi$ is $G$-equivariant on a
$G$-invariant Borel sets of full measure. This implies that the Boolean $\sigma$-algebra
$\Phi^*(\mathcal{B}(Y))$ consists of $G$-invariant elements, i.e.,
$$
\Phi^*(\mathcal{B}(Y))\subset P_G^*(\mathcal{B}(X\ec G)).
$$
Let $\bar\phi:\mathcal{B}(Y)\to \mathcal{B}(X\ec G)$ be the corresponding embedding.
By (ii) there exists a $G$-map $\bar\Phi:X\ec G \to Y$ such that $\bar\Phi^*=\phi$.
Then $P_G^*\circ \bar \Phi^*= (\bar \Phi\circ P_G)^*= \Phi^*$, and it follows from (i)
that $\bar\Phi \circ P_G=\Phi$ on a conull set, as required.

%To prove (2) we consider $\Phi$ as an element of $\hbox{Mor}(X,Y)$,
%which is the space of Borel maps $X\to Y$ identified if they agree on a conull set,
%and introduce an action of $G$ on $\hbox{Mor}(X,Y)$ given by
%$(g\cdot \Phi)(x)=\Phi(g^{-1}x)$.
%Then $\hbox{Mor}(X,Y)$ can be equipped with a structure of Polish space so
%that the above action of $G$ is continuous (see \cite[Proposition~2.16]{FMW}).
%Since the  element $\Phi\in \hbox{Mor}(X,Y)$ is fixed
%by the subgroups $H_i$, $i\in I$, it is fixed by $H$ as well. This implies (2).

Let $H$ be a closed normal subgroup of $G$.
Since $P_H^*(\mathcal{B}(X\ec H))$ is the algebra of $H$-invariant elements,
it invariant under $G$. This defines an action of $G$ on $\mathcal{B}(X\ec H)$ such that
the map $P_H^*:\mathcal{B}(X\ec H)\to\mathcal{B}(X)$ is $G$-equivariant.
By (iii), the $G$-action on $\mathcal{B}(X\ec H)$ comes from a structure
of a $G$-space on $X\ec H$.
Since for every $g\in G$, we have $g^*\circ P_H^*=P_H^*\circ g^*$, it follows from (i) that
$P_G\circ g=g\circ P_G$ on a conull set. This completes the proof.
\end{proof}

\begin{cor} \label{c:ec}
Let $G$ be an lcsc group, let $H$ be its closed subgroup, and let $\overline{H}^{MT}$ be the Mautner envelope of $H$ as in \S\ref{s:mautner}.
Let  $(X,\xi)$ be a $G$-space with a $G$-invariant probability measure.
Then the natural map $X\ec H\to X\ec \overline{H}^{MT}$ is a Borel isomorphism.
\end{cor}

\begin{proof}
To prove the corollary we need to show that every
Borel subset $B$ of $X$ such that $\mu(B\triangle hB)=0$ for $h\in H$,
we have $\mu(B\triangle hB)=0$ for $h\in \overline{H}^{MT}$.
Indeed, the characteristic function of $B$ is an
${H}$-invariant vector in $L^2(X,\mu)$,
hence it is also invariant by $\overline{H}^{MT}$.
\end{proof}

%In particular, we have
%
%\begin{lem}\label{l:hw}
%Let $G$ be an $S$-algebraic group, and $R$ an $S$-algebraic subgroup of $G$.
%Let $(X,\xi)$ be $G$-space with invariant measure.
%Denote by $N$ the almost algebraic kernel of the Mautner envelope of $R$.
%Then the natural map
%$X\ec R \to (X\ec N)\ec R$ is an isomorphism.
%\end{lem}

\begin{proof}
Since $N<\MT{H}$, the natural map $X\ec \MT{H}\to (X\ec N)\ec \MT{H}$
is an isomorphism.
Since by Corollary~\ref{c:ec},
$X\ec H\simeq  X\ec \MT{H}$ and $(X\ec N)\ec H\simeq (X\ec N)\ec \MT{H}$,
this implies the claim.
\end{proof}

% section measure_class_preserving_actions (end)	
\end{appendix}

\end{document}

#####################################################################

\end{document}